\documentclass{amsart}
\usepackage[latin9]{inputenc}
\usepackage{color}
\usepackage{verbatim}
\usepackage{textcomp}
\usepackage{mathtools}
\usepackage{amstext}
\usepackage{amsthm}
\usepackage{amssymb}
\usepackage[unicode=true,
 bookmarks=false,
 breaklinks=false,pdfborder={0 0 1},backref=false,colorlinks=false]
 {hyperref}
\hypersetup{
 bookmarksnumbered,plainpages}
\usepackage{breakurl}

\makeatletter
\numberwithin{equation}{section}
\numberwithin{figure}{section}
\theoremstyle{plain}
\newtheorem{thm}{\protect\theoremname}
\theoremstyle{plain}
\newtheorem{lem}[thm]{\protect\lemmaname}
\theoremstyle{definition}
\newtheorem{defn}[thm]{\protect\definitionname}
\theoremstyle{remark}
\newtheorem{rem}[thm]{\protect\remarkname}
\theoremstyle{plain}
\newtheorem{prop}[thm]{\protect\propositionname}
\theoremstyle{plain}
\newtheorem{cor}[thm]{\protect\corollaryname}

\usepackage{amsfonts}
\usepackage{float}
\usepackage{version}
\allowdisplaybreaks
\makeatother

\providecommand{\corollaryname}{Corollary}
\providecommand{\definitionname}{Definition}
\providecommand{\lemmaname}{Lemma}
\providecommand{\propositionname}{Proposition}
\providecommand{\remarkname}{Remark}
\providecommand{\theoremname}{Theorem}

\begin{document}
\global\long\def\perpl{\perp_{L}}
 \global\long\def\perpr{\perp_{R}}

\thanks{This note was written while the first author was member of the \textquotedblleft National
Group for Algebraic and Geometric Structures, and their Applications\textquotedblright{}
(GNSAGA-INdAM). We would like to thank Michela Ceria for meaningful
discussions on the topics treated in the present paper. We are also
in debt with Lea Terracini for her contribution in the computations
of the indecomposable ideals of the Taft algebra and their orthogonals.}


\title{On Indecomposable Ideals Over Some Algebras}

\author{Alessandro Ardizzoni }

\address{University of Turin, Department of Mathematics \textquotedblleft G.
Peano\textquotedblright , via Carlo Alberto 10, I-10123 Torino, Italy.}

\email{alessandro.ardizzoni@unito.it }

\urladdr{sites.google.com/site/aleardizzonihome}

\author{Fabio Stumbo}

\address{University of Ferrara, Department of Mathematics and Computer Science,
Via Ma\-chiavelli 30, Ferrara, I-44121, Italy.}

\email{f.stumbo@unife.it}
\begin{abstract}
In this paper we investigate a family of algebras endowed with a suitable
non-degenerate bilinear form that can be used to define two different
notions of dual for a given right ideal. We apply our results to the
classification of the right ideals and their duals in the cyclic group
algebra, in the Taft algebra and in another example of Hopf algebra
arising as bosonization.
\end{abstract}

\maketitle
\subjclass{[}2010{]}{Primary 16T05; Secondary 15A63, 94B15, 94B05}

\tableofcontents{}

\section*{Introduction}

Fix a base field $\Bbbk$. Recall that a linear code of length $n\geq2$
is a vector subspace of $\Bbbk^{n}$. A linear code $C$ is called
cyclic whenever $\left(c_{0},c_{1},\cdots,c_{n-1}\right)\in C$ implies
$\left(c_{n-1},c_{0},c_{1},\cdots,c_{n-2}\right)\in C.$ By considering
the cyclic group algebra $\Bbbk\left\langle x\right\rangle :=\Bbbk\left[X\right]/\left(X^{n}-1\right)$,
where we set $x:=X+\left(X^{n}-1\right)$, then the assignments
\begin{equation}
\Bbbk\left\langle x\right\rangle \ni c_{0}+c_{1}x+\cdots+c_{n-1}x^{n-1}\longleftrightarrow\left(c_{0},c_{1},\cdots,c_{n-1}\right)\in\Bbbk^{n}\label{eq: KCn}
\end{equation}
yield a bijective correspondence between ideals in $\Bbbk\left\langle x\right\rangle $
and cyclic codes of length $n$. Moreover, if $g$ is the generator
polynomial for a cyclic code, then the dual code corresponds, via
the map \eqref{eq: KCn}, to the orthogonal of the ideal $\left(g\right)$
with respect to the usual scalar product defined by $\left\langle x^{i},x^{j}\right\rangle =\delta_{i,j}$.

Since $\Bbbk\left\langle x\right\rangle $ is the group algebra over
the cyclic group of order $n$, i.e. $\left\langle x\right\rangle :=\left\{ 1,x,\ldots,x^{n-1}\right\} $,
it is in particular a Hopf algebra. This constitutes a link between
the study of the ideals in Hopf algebras and Coding Theory. It is
then natural to look at a right ideal in a Hopf algebra as a sort
of \textquotedblleft Hopf code\textquotedblright . Following this
idea, a characterization of all projective indecomposable ideals in
the Taft algebra was given in \cite{CGL}.

The particular description of dual codes in the Hopf algebra $\Bbbk\left\langle x\right\rangle $
places the notion of orthogonal with respect to a suitable non-degenerate
bilinear form as the proper counterpart of the concept of dual code.

We will show that both the cyclic group algebra and the Taft algebra
come out to be endowed with a non-degenerate bilinear form that we
will call monomial and that can be used to compute explicitly the
indecomposable right ideals and their orthogonals.

More generally in this paper we investigate a family of algebras endowed
with a monomial non-degenerate bilinear form: this form can be used
to define two different notions of dual for a given right ideal. As
an application, we recover the classification of the right ideals
and their duals in the cyclic group and Taft algebras and provide
the same classification for another example of Hopf algebra arising
as a bosonization in \cite{CDMM}.

\smallskip{}

The paper is organized as follows.

In Section \ref{sec:Bilinear-forms} we introduce and characterize
the concept of monomial bilinear form on a finite-dimensional vector
space $V$. This type of form is the main tool in our paper. In Lemma
\ref{lem:orthog}, we study the behaviour of orthogonals of subspaces
of $V$ spanned by a subset of the basis, with respect to a monomial
bilinear form.

In Section \ref{sec:Indec-ideals} we introduce the $\Bbbk$-algebra
$\Bbbk\left(\omega,N\right)$ presented by generators and relations.
It can also be introduced as a quotient of an Ore extension with zero
derivation. In Theorem \ref{thm:indecomposableGold} we provide an
irredundant set of representatives of indecomposable right ideals
in $\Bbbk\left(\omega,N\right)$. We also give an explicit description
of all indecomposable right ideals therein. This is of interest from
the Coding Theory point of view because isomorphic ideals needs not
to give rise to equivalent codes. Then we attach to $\Bbbk\left(\omega,N\right)$
a particular bilinear form which we prove to be monomial in Lemma
\ref{lem:scalarcan}. In Theorem \ref{thm:orthog-indec}, we describe
explicitly the orthogonals, with respect to this form, of the representatives
of indecomposable right ideals mentioned above.

In Section \ref{sec:Hopf-algebra-codes} we turn our attention to
Hopf algebras. A peculiar result in Hopf theory is the Structure Theorem
for Hopf Modules. Given a finite-dimensional Hopf algebra $H$, the
structure theorem yields a right $H$-linear isomorphism $\phi:H\to H^{*}$.
Surprisingly, $\phi$ comes out to generalize the map \eqref{eq: KCn}
giving the correspondence mentioned above for cyclic codes; thus $\phi$
looks like a natural tool to generalize cyclic codes to Hopf algebras.
In general $\phi$ yields a bijective correspondence between right
ideals in $H$ and right $H$-submodules of $H^{*}$ (which play the
role of generalized cyclic codes). We attach to $\phi$ a specific
non-degenerate bilinear form and use it to define two different notions
of dual for a given right ideal: both comes out to be right ideals
as well. We investigate this form in the case when $H$ or its antipode
$S$ satisfy some properties, such as $H$ being cosemisimple in Lemma
\ref{lem:Involutory}.

In Section \ref{sec:Examples-and-applications} we collect and investigate
the main examples and applications of our results.

First we recover the cyclic group algebra as a trivial case of $\Bbbk\left(\omega,N\right)$
taking $\omega=\mathrm{Id}$ and $N=1$ in Subsection \ref{subsec:Cyclic-group-algebra}.

Then, in Subsection \ref{subsec:Taft-algebra}, we apply our machinery
to the Taft algebra. This is achieved regarding the Taft algebra as
an algebra of the form $\Bbbk\left(\omega,N\right)$ in Lemma \ref{lem:formuTaft}
and showing in Lemma \ref{lem:scalar-1} that the non-degenerate bilinear
form attached to $\phi$ is one of the canonical monomial forms we
investigated on $\Bbbk\left(\omega,N\right)$. As a consequence we
recover all indecomposable right ideals for the Taft algebra in Theorem
\ref{thm:indecomposable} and the corresponding orthogonals in Theorem
\ref{thm:orthogonal}.

Finally in Subsection \ref{subsec:Another-example}, we consider a
$24$-dimensional Hopf algebra arising as a bosonization described
in \cite{CDMM}. Following the lines of the previous examples, we
show in Lemma \ref{lem:formCDMM} that also this algebra is of the
form $\Bbbk\left(\omega,N\right)$ and in Lemma \ref{lem:scalar-2}
that the non-degenerate bilinear form attached to $\phi$ is one of
the canonical monomial forms we investigated on $\Bbbk\left(\omega,N\right)$.
Also in this case we obtain a complete description of indecomposable
right ideals in Theorem \ref{thm:indecomposableCDMM} and their orthogonals
in \ref{thm:orthogonal-CDMM}. This classification is new at the best
of our knowledge.

\section{Bilinear forms\label{sec:Bilinear-forms}}

In this section we deal with general results concerning vector spaces.

Let $V$ be a finite-dimensional vector space and let $\left\langle -,-\right\rangle :V\times V\to\Bbbk$
be a bilinear form on $V$. The form is non-degenerate if and only
if the linear map $\phi:V\to V^{*}$, defined by setting $\phi\left(v\right)=\left\langle v,-\right\rangle $
for every $v\in V$, is an isomorphism.

If $\left\langle -,-\right\rangle $ is non-degenerate there is a
unique linear map $\gamma:V\to V$ such that
\begin{equation}
\left\langle x,y\right\rangle =\left\langle y,\gamma\left(x\right)\right\rangle ,\text{ for all }x,y\in V,\label{eq:gamma}
\end{equation}
The map $\gamma$ is necessarily injective whence invertible as $V$
is finite-dimensional. We call it the \textit{Nakayama isomorphism}
(if $V$ is an algebra and the form is associative, i.e. $\left\langle xy,z\right\rangle =\left\langle x,yz\right\rangle $
for all $x,y,z\in V$, then $\gamma$ becomes an algebra automorphism
known as the Nakayama automorphism).

Given a linear subspace $W$ of $V$, we define the left and right
orthogonals of $W$ as the vector subspaces
\begin{align*}
W^{\perp_{L}} & :=\left\{ x\in V\mid\left\langle x,y\right\rangle =0,\forall y\in W\right\} ,\\
W^{\perp_{R}} & :=\left\{ x\in V\mid\left\langle y,x\right\rangle =0,\forall y\in W\right\} .
\end{align*}
We will simply write $W^{\perp}$when $W^{\perp_{L}}=W^{\perp_{R}}$.
In particular this happens if, for every $x,y\in V$, one has $\left\langle x,y\right\rangle =0\Leftrightarrow\left\langle y,x\right\rangle =0$.
\begin{lem}
\cite[Section 6.1]{Jacobson}\label{lem:SwA.1} Assume $\left\langle -,-\right\rangle $
is non-degenarate and let $U$ and $W$ be subspaces of $V$. Then
\end{lem}

\begin{itemize}
\item if $U\subseteq W$ then $W^{\perp_{L}}\subseteq U^{\perp_{L}}$ and
$W^{\perp_{R}}\subseteq U^{\perp_{R}}$;
\item $\left(U+W\right)^{\perp_{L}}=U^{\perp_{L}}\cap W^{\perp_{L}}$ and
$\left(U+W\right)^{\perp_{R}}=U^{\perp_{R}}\cap W^{\perp_{R}}$;
\item $\left(U\cap W\right)^{\perp_{L}}=U^{\perp_{L}}+W^{\perp_{L}}$ and
$\left(U\cap W\right)^{\perp_{R}}=U^{\perp_{R}}+W^{\perp_{R}}$;
\item $\mathrm{dim}_{\Bbbk}W^{\perp_{L}}=\mathrm{dim}_{\Bbbk}V-\mathrm{dim}_{\Bbbk}W=\mathrm{dim}_{\Bbbk}W^{\perp_{R}}$;
\item $W^{\perp_{R}\perp_{L}}=W=W^{\perp_{L}\perp_{R}}$.
\end{itemize}
\begin{lem}
\label{lem:gamDuals}Let $\gamma:V\to V$ be a $\Bbbk$-linear isomorphism
such that \eqref{eq:gamma} holds true. Then $W^{\perp_{R}}=\gamma\left(W^{\perp_{L}}\right)$
for every subspace $W$ of $V$.
\end{lem}

\begin{proof}
Let $x\in V$. Then $\gamma\left(x\right)\in W^{\perp_{R}}$ if and
only if $\left\langle y,\gamma\left(x\right)\right\rangle =0$, for
every $y\in W$. By \eqref{eq:gamma} this is equivalent to $\left\langle x,y\right\rangle =0$,
for every $y\in W$ i.e. $x\in W^{\perp_{L}}$. Thus $\gamma\left(x\right)\in W^{\perp_{R}}$
if and only if $x\in W^{\perp_{L}}$. Since $\gamma$ is invertible
we conclude.
\end{proof}
As usual let us denote by $\mathcal{S}_{n}$ the group of permutations
of the set $\left\{ 1,\dots,n\right\} $ for every $n\geq1$.
\begin{defn}
Given a $\Bbbk$-vector space $V$ with basis $\mathcal{B}=\left\{ v_{1},\dots,v_{n}\right\} $,
a \textit{right monomial transformation of $V$ with respect to $\mathcal{B}$
}(see e.g. \cite[page 12]{Wood}) is a $\Bbbk$-linear map $T:=T\left(\sigma,k_{i}\right):V\to V$
such that $T\left(v_{i}\right)=k_{i}v_{\sigma\left(i\right)}$ where
$\sigma\in\mathcal{S}_{n}$ and $k_{i}\in\Bbbk\setminus\left\{ 0\right\} $
for every $i\in\left\{ 1,\dots,n\right\} $.
\end{defn}

\begin{lem}
\label{lem:monomial}The following are equivalent for a basis $\mathcal{B}=\left\{ v_{1},\dots,v_{n}\right\} $
of $V$.
\begin{enumerate}
\item There is $\sigma\in\mathcal{S}_{n}$ and non-zero elements $d_{i}\in\Bbbk$
such that, for every $i,j\in\left\{ 1,\dots,n\right\} $, one has
$\left\langle v_{i},v_{j}\right\rangle =d_{i}\delta_{\sigma\left(i\right),j}$.
\item There is $\sigma\in\mathcal{S}_{n}$ and non-zero elements $d_{i}\in\Bbbk$
such that $\phi\left(v_{i}\right)=d_{i}v_{\sigma\left(i\right)}^{*}$,
where $v_{j}^{*}$ is the dual basis element defined by $v_{j}^{*}\left(v_{i}\right)=\delta_{i,j}$.
\item The matrix of $\left\langle -,-\right\rangle $ relative to $\mathcal{B}$
is of form $D\cdot P$ where $D$ is a non-singular diagonal matrix
and $P$ is a permutation matrix.
\item The $\Bbbk$-linear map $T:\Bbbk^{n}\to\Bbbk^{n}:e_{i}\mapsto\sum_{j=1}^{n}\left\langle v_{i},v_{j}\right\rangle e_{j}$
(i.e. the linear map attached to the matrix of $\left\langle -,-\right\rangle $
relative to $\mathcal{B}$) is a right monomial transformation of
$\Bbbk^{n}$ with respect to the canonical basis $\left\{ e_{1},\dots,e_{n}\right\} $.
\end{enumerate}
\end{lem}

\begin{proof}
$\left(1\right)\Leftrightarrow\left(2\right).$ This equivalence follows
from the equalities $\phi\left(v_{i}\right)\left(v_{j}\right)=\left\langle v_{i},v_{j}\right\rangle $
and $\left(d_{i}v_{\sigma\left(i\right)}^{*}\right)\left(v_{j}\right)=d_{i}\delta_{\sigma\left(i\right),j}$.

$\left(1\right)\Leftrightarrow\left(3\right).$ Take $D=\mathrm{diag}\left(d_{1},\dots,d_{n}\right)$
and let $P$ be the matrix whose $\left(i,j\right)$-entry is given
by $\delta_{\sigma\left(i\right),j}$. Then the $\left(i,j\right)$-entry
of $D\cdot P$ is exactly $d_{i}\delta_{\sigma\left(i\right),j}$.

$\left(1\right)\Leftrightarrow\left(4\right).$ By definition $T$
is a right monomial transformation if and only if there is a permutation
$\sigma$ of $\left\{ 1,\dots,n\right\} $ and non-zero elements $d_{i}\in\Bbbk$
such that $T\left(e_{i}\right)=d_{i}e_{\sigma\left(i\right)}.$
\end{proof}
\begin{defn}
We say that the bilinear form $\left\langle -,-\right\rangle $ is
\textit{monomial} (with respect to $\mathcal{B}$) if one of the equivalent
conditions of Lemma \ref{lem:monomial} holds. Note that a monomial
bilinear form is necessarily non-degenarate since $\phi\left(v_{i}\right)=d_{i}v_{\sigma\left(i\right)}^{*}$
for every $i\in\left\{ 1,\dots,n\right\} $.
\end{defn}

\begin{rem}
\label{rem:gamma} Let $V$ be a vector space with basis $\mathcal{B}=\left\{ v_{1},\dots,v_{n}\right\} $.
The following are equivalent for $\tau\in\mathcal{S}_{n}$ and non-zero
elements $c_{i}\in\Bbbk$.
\begin{enumerate}
\item For every $i,j\in\left\{ 1,\dots,n\right\} $, one has $\left\langle v_{i},v_{j}\right\rangle =c_{i}\left\langle v_{j},v_{\tau\left(i\right)}\right\rangle $.
\item One has that \eqref{eq:gamma} holds true where $\gamma:=T\left(\tau,c_{i}\right):V\to V$
is the linear map defined by $v_{i}\mapsto c_{i}v_{\tau\left(i\right)}$.
Note that $\gamma$ is necessarily invertible.
\end{enumerate}
\end{rem}

\begin{rem}
Let $\mathcal{B}=\left\{ v_{1},\dots,v_{n}\right\} $ be a basis of
$V$ and let $W$ be a subspace of $V$. Assume there is a right monomial
transformation $\gamma=T\left(\tau,c_{i}\right)$ such that \eqref{eq:gamma}
holds true. By \cite[Proposition 6.1]{Wood}, we have that $\gamma$
is an isometry of $V$ with respect to the basis $\mathcal{B}$. By
Lemma \ref{lem:gamDuals} we also have $W^{\perp_{R}}=\gamma\left(W^{\perp_{L}}\right)$.
Then $W^{\perp_{R}}$ and $W^{\perp_{L}}$ are equivalent codes in
the sense of \cite[page 565]{Wood}.
\end{rem}

In the following result we describe the connection between the left
and right orthogonals in case of vector subspaces spanned by suitable
basis elements.
\begin{lem}
\label{lem:orthog} Let $\mathcal{B}=\left\{ v_{1},\dots,v_{n}\right\} $
be a basis of $V$ and let $W_{B}:=\mathrm{Span}_{\Bbbk}\left\{ v_{i}\mid i\in B\right\} $
for some $B\subseteq\left\{ 1,\dots,n\right\} $.

1) Assume there is a right monomial transformation $\gamma=T\left(\tau,c_{i}\right)$
such that \eqref{eq:gamma} holds true. Then $W_{B}^{\perp_{R}}=W_{\tau\left(B\right)}^{\perp_{L}}$
where $\tau\left(B\right):=\left\{ v_{\tau\left(i\right)}\mid i\in B\right\} $.

2) If $\left\langle -,-\right\rangle $ is monomial with respect to
$\mathcal{B}$, then the hypothesis of 1) holds for $\tau:=\sigma^{2}$.
\end{lem}

\begin{proof}
1) From the assumption one gets $\left\langle v_{i},v\right\rangle =c_{i}\left\langle v,v_{\tau\left(i\right)}\right\rangle $
for every $v\in V$. Since $c_{i}\neq0$ we deduce that $\left\langle v_{i},v\right\rangle =0$
if and only if $\left\langle v,v_{\tau\left(i\right)}\right\rangle =0$.
Now $v\in W_{B}^{\perp_{R}}$ if and only if $\left\langle v_{i},v\right\rangle =0$
for every $i\in B$ if and only if, by the foregoing, $\left\langle v,v_{\tau\left(i\right)}\right\rangle =0$
for every $i\in B$ if and only $v\in W_{\tau\left(B\right)}^{\perp_{L}}$.

2) By Lemma \ref{lem:monomial}, there exists a permutation $\sigma$
of $\left\{ 1,\dots,n\right\} $ and $\left\langle v_{i},v_{j}\right\rangle =d_{i}\delta_{\sigma\left(i\right),j}$
for some non-zero $d_{i}\in\Bbbk$. Then we can take $c_{i}:=\frac{d_{i}}{d_{\sigma\left(i\right)}}$
and compute
\[
\left\langle v_{j},c_{i}v_{\sigma^{2}\left(i\right)}\right\rangle =c_{i}d_{j}\delta_{\sigma\left(j\right),\sigma^{2}\left(i\right)}=c_{i}d_{j}\delta_{\sigma\left(i\right),j}=c_{i}d_{\sigma\left(i\right)}\delta_{\sigma\left(i\right),j}=d_{i}\delta_{\sigma\left(i\right),j}=\left\langle v_{i},v_{j}\right\rangle .
\]
\end{proof}
\begin{rem}
In the case of the Taft algebra $A$ we will provide a particular
$\left\langle -,-\right\rangle $ which satisfies the conditions in
Lemma \ref{lem:orthog} but not $\left\langle x,y\right\rangle =0\Leftrightarrow\left\langle y,x\right\rangle =0$
for every $x,y\in A$, cf. \cite[Theorem 6.2]{Jacobson}.

The proof of the following result is straightforward. We keep it for
the reader's sake.
\end{rem}

\begin{prop}
\label{prop:action}Let $A$ be an algebra. Let $V$ be a vector space
endowed with a non-degenarate bilinear form $\left\langle -,-\right\rangle $.
Let $W$ be a subspace of $V$.

1) If $V$ has a left $A$-module structure, then $V$ has (necessarily
unique) right $A$-module structures $\vartriangleleft$ and $\blacktriangleleft$
defined for every $x,y\in V,c\in A$, respectively, by
\[
\left\langle x\vartriangleleft c,y\right\rangle =\left\langle x,cy\right\rangle \qquad\text{and}\qquad\left\langle x,y\blacktriangleleft c\right\rangle =\left\langle cx,y\right\rangle .
\]

Moreover $\left\{ c\in A\mid\gamma\left(cx\right)=c\gamma\left(x\right),\forall x\in V\right\} =\left\{ c\in A\mid y\vartriangleleft c=y\blacktriangleleft c,\forall y\in V\right\} $.
In particular $\gamma$ is left $A$-linear if and only if $\vartriangleleft=\blacktriangleleft$.

2) If $V$ has a right $A$-module structure, then $V$ has (necessarily
unique) left $A$-module structures $\vartriangleright$ and $\blacktriangleright$
defined for every $x,y\in V,c\in A$, respectively, by
\[
\left\langle c\vartriangleright x,y\right\rangle =\left\langle x,yc\right\rangle \qquad\text{and}\qquad\left\langle x,c\blacktriangleright y\right\rangle =\left\langle xc,y\right\rangle .
\]

Moreover $\left\{ c\in A\mid\gamma\left(xc\right)=\gamma\left(x\right)c,\forall x\in V\right\} =\left\{ c\in A\mid c\vartriangleright y=c\blacktriangleright y,\forall y\in V\right\} $.
In particular $\gamma$ is right $A$-linear if and only if $\vartriangleright=\blacktriangleright$.

For $c\in A$, we have
\begin{align*}
\left(V\vartriangleleft c\right)^{\perp_{R}} & =\left(V\blacktriangleleft c\right)^{\perp_{L}}=\mathrm{Ann}_{_{\bullet\negmedspace}V}\left(c\right):=\left\{ v\in V\mid cv=0\right\} ,\\
\left(c\vartriangleright V\right)^{\perp_{R}} & =\left(c\blacktriangleright V\right)^{\perp_{L}}=\mathrm{Ann}_{V_{\bullet}}\left(c\right):=\left\{ v\in V\mid vc=0\right\} .
\end{align*}
Let $a,b\in A$ be such that $ab=1$. Then

\begin{align*}
\left(bW\right)^{\perp_{L}}=W^{\perp_{L}}\vartriangleleft a\qquad & \text{and}\qquad\left(W\vartriangleleft a\right)^{\perp_{R}}=bW^{\perp_{R}}\\
\left(W\blacktriangleleft a\right)^{\perp_{L}}=bW^{\perp_{L}}\qquad & \text{and}\qquad\left(bW\right)^{\perp_{R}}=W^{\perp_{R}}\blacktriangleleft a\\
\left(Wa\right)^{\perp_{L}}=b\vartriangleright W^{\perp_{L}}\qquad & \text{and}\qquad\left(b\vartriangleright W\right)^{\perp_{R}}=W^{\perp_{R}}a\\
\left(b\blacktriangleright W\right)^{\perp_{L}}=W^{\perp_{L}}a\qquad & \text{and}\qquad\left(Wa\right)^{\perp_{R}}=b\blacktriangleright W^{\perp_{R}}
\end{align*}
\end{prop}

\begin{proof}
We just prove the assertions involving $\vartriangleleft$, the other
ones being similar. Since $\phi:V\to V^{*}$ is an isomorphism and
$V^{*}$ is a right $A$-module we can endow $V$ with a unique right
$A$-module structure such that $\phi$ is right $A$-linear. The
right $A$-linearity of $\phi$ is equivalent to $\left\langle x\vartriangleleft c,y\right\rangle =\left\langle x,cy\right\rangle $,
for every $x,y\in V,c\in A$.

We compute

\begin{align*}
\left\langle x,y\blacktriangleleft c\right\rangle  & =\left\langle cx,y\right\rangle =\left\langle y,\gamma\left(cx\right)\right\rangle ;\\
\left\langle x,y\vartriangleleft c\right\rangle  & =\left\langle y\vartriangleleft c,\gamma\left(x\right)\right\rangle =\left\langle y,c\gamma\left(x\right)\right\rangle .
\end{align*}
Then$\gamma\left(cx\right)=c\gamma\left(x\right)$ for every $x\in V$
if and only if $y\blacktriangleleft c=y\vartriangleleft c$ for every
$y\in V$.

From $\left\langle x\vartriangleleft c,y\right\rangle =\left\langle x,cy\right\rangle $
it is clear that $\left(V\vartriangleleft c\right)^{\perp_{R}}=\mathrm{Ann}_{_{\bullet\negmedspace}V}\left(c\right)$.

Then
\[
\left\langle x\vartriangleleft a,by\right\rangle =\left\langle x,aby\right\rangle =\left\langle x,y\right\rangle
\]
so that $x\in W^{\perp_{L}}$ if and only if $x\vartriangleleft a\in\left(bW\right)^{\perp_{L}}$.
Thus $\left(bW\right)^{\perp_{L}}=W^{\perp_{L}}\vartriangleleft a$.
Similarly $\left(W\vartriangleleft a\right)^{\perp_{R}}=bW^{\perp_{R}}$.
\end{proof}

\section{Indecomposable ideals in $\Bbbk\left(\omega,N\right)$\label{sec:Indec-ideals}}

Let us consider the main example of algebra we will deal with. Let
$N\geq2$ be an integer, let $S$ be a finite set of cardinality at
least $N$ and let $\omega$ be a permutation on $S$.

Consider the $\Bbbk$-algebra $A=\Bbbk\left(\omega,N\right)$ generated
by $\left\{ e_{s},x\mid s\in S\right\} $ with relations, for every
$s,t\in S$,
\begin{align}
e_{s}e_{t} & =\delta_{s,t}e_{t},\qquad1_{A}=\sum_{s\in S}e_{s},\qquad x^{N}=0,\qquad e_{s}x=xe_{\omega\left(s\right)}.\label{eq:relA}
\end{align}
It is clear that $A$ has basis
\[
\mathcal{B}:=\left\{ x^{n}e_{s}\mid0\leq n\leq N-1,s\in S\right\} =\left\{ e_{t}x^{n}\mid0\leq n\leq N-1,t\in S\right\} ,
\]
 where the last equality follows by the relation $e_{s}x=xe_{\omega\left(s\right)}$,
for all $s\in S$.

Let $R=\Bbbk\left[x\mid x^{N}=0\right]$ be the subalgebra of $A$
generated by $x$.

Set also $H:=\Bbbk\left[e_{s}\mid s\in S\right]$.

Following the lines of \cite{CGL}, for $s\in S,t=0,\ldots,N-1$ we
set $N_{s,t}:=e_{s}x^{t}R.$
\begin{rem}
The algebra $A$ given above can be constructed as follows. Let $H=\Bbbk^{S}=\mathrm{Map}\left(S,\Bbbk\right)$
be the algebra of functions on the set $S$. Define $e_{s}:S\to\Bbbk$
by setting $e_{s}\left(t\right):=\delta_{s,t}$. Then $\left\{ e_{s}\mid s\in S\right\} $
is a basis of $H$ and for every $s,t\in S$, one has
\begin{align*}
e_{s}e_{t} & =\delta_{s,t}e_{t},\qquad1_{A}=\sum_{s\in S}e_{s}.
\end{align*}
Define the map $\varphi:H\to H$ by setting $\varphi\left(e_{s}\right):=e_{\omega^{-1}\left(s\right)}$.
It is easy to check that $\varphi$ is an algebra map. Consider the
Ore extension $H\left[X,\varphi\right]$ with zero derivation i.e.
$H\left[X\right]$ as an abelian group with multiplication induced
by $Xe_{s}=\varphi\left(e_{s}\right)X$. Then $A=\Bbbk\left(\omega,N\right)$
is the quotient of $H\left[X,\varphi\right]$ modulo the two-sided
ideal $\left\langle X^{N}\right\rangle $ generated by $X^{N}$, with
notation $x:=X+\left\langle X^{N}\right\rangle $ and by identifying
$e_{s}$ with its class modulo $\left\langle X^{N}\right\rangle $.
\end{rem}

Denote by $\mathcal{L}\left(M_{R}\right)$ the set of right $R$-submodules
of a given right $R$-module $M_{R}.$ Let $\mathcal{L}_{\mathrm{in}}\left(M_{R}\right)$
denote the set of right $R$-submodules which are indecomposable.
The first assertion of the following result was already proved for
Taft algebras in \cite[Section 2]{CGL}, see also \cite[Theorem 2.5]{CVZ}.
\begin{thm}
\label{thm:indecomposableGold}Consider the algebra $A=\Bbbk\left(\omega,N\right)$
and let $R$ be the subalgebra of $A$ generated by $x.$ Then the
$N_{s,t}$'s form an irredundant set of representatives of $\mathcal{L}_{\mathrm{in}}\left(A_{A}\right)$
and
\[
\mathcal{L}_{\mathrm{in}}\left(A_{A}\right)=\left\{ \left(1+rx\right)N_{s,t}\mid r\in R,s\in S,0\leq t\leq N-1\right\} .
\]
\end{thm}

\begin{proof}
We have already observed that $A$ has basis
\[
\mathcal{B}:=\left\{ x^{n}e_{s}\mid0\leq n\leq N-1,s\in S\right\} =\left\{ e_{t}x^{n}\mid0\leq n\leq N-1,t\in S\right\} .
\]

Let us show that $A$ is a serial ring. Since $1_{A}=\sum_{s\in S}e_{s}$
and the $e_{s}$'s are orthogonal idempotents, we have that $A=\bigoplus_{s\in S}e_{s}A.$
Thus each direct summand $e_{s}A$ of $A$ is projective too. By \eqref{eq:relA},
we deduce $e_{s}A=e_{s}R$ and hence $\mathcal{L}\left(e_{s}A_{A}\right)=\mathcal{L}\left(e_{s}R_{R}\right)$.
The map
\[
g_{s}:R\rightarrow e_{s}R:r\mapsto e_{s}r
\]
is an isomorphism of right $R$-modules so that it induces a bijection
\[
\mathcal{L}\left(R_{R}\right)\rightarrow\mathcal{L}\left(e_{s}R_{R}\right):I\mapsto g_{s}\left(I\right)=e_{s}I.
\]
Since $\mathcal{L}\left(R_{R}\right)=\left\{ x^{t}R\mid0\leq t\leq N\right\} $
we get
\[
\mathcal{L}\left(e_{s}A_{A}\right)=\mathcal{L}\left(e_{s}R_{R}\right)=\left\{ e_{s}x^{t}R\mid0\leq t\leq N\right\} =\left\{ N_{s,t}\mid0\leq t\leq N-1\right\} \cup\left\{ 0\right\} .
\]
Since $N_{s,t}\subseteq N_{s,w}$ for $t\geq w,$ we deduce that $\mathcal{L}\left(e_{s}A_{A}\right)$
is totally ordered. Thus $e_{s}A$ is uniserial for every $s\in S$
and hence $A$ is a right serial ring. Similarly one proves it is
a left serial ring. Thus it is a serial ring as claimed.

Note that, since uniserial implies indecomposable, we get that
\[
\mathcal{L}_{\mathrm{in}}\left(e_{s}A_{A}\right)=\mathcal{L}\left(e_{s}A_{A}\right)\setminus\left\{ 0\right\} =\left\{ N_{s,t}\mid0\leq t\leq N-1\right\} .
\]

Let $M\in\mathcal{L}_{\mathrm{in}}\left(A_{A}\right).$ Apply \cite[Theorem 3.29]{Fa}
to $P=A_{A}$. Then there is a decomposition $A_{A}=P_{1}\oplus\cdots\oplus P_{l}$
into uniserial projective modules such that $M=\left(M\cap P_{1}\right)\oplus\cdots\oplus\left(M\cap P_{l}\right).$
Since $M$ is indecomposable, we get $M=M\cap P_{w}$ for some $w.$
Hence $M\subseteq P_{w}.$ Since each $P_{i}$ is uniserial, it is
in particular indecomposable.

Note that $\mathrm{End}_{A}\left(e_{s}A\right)\cong e_{s}Ae_{s}\overset{\eqref{eq:relA}}{=}e_{s}Re_{s}$.
Since $R=\Bbbk\left[x\mid x^{N}=0\right]$ is a local ring, we have
that $e_{s}R$ is a local right $R$-module so that, by \cite[Theorem 1.11]{Fa},
we get that $e_{s}Re_{s}$ is a local ring.

By Krull-Schmidt-Remak-Azumaya Theorem (cf. \cite[Theorem 2.12]{Fa}),
the two decompositions $\bigoplus_{s\in S}e_{s}A=A_{A}=P_{1}\oplus\cdots\oplus P_{l}$
are necessarily isomorphic so that $P_{w}\cong e_{s}A$ for some $s\in S.$
Let $f_{w}:e_{s}A\rightarrow P_{w}$ this isomorphism of right $A$-modules.

We have $f_{w}\left(e_{s}\right)=f_{w}\left(e_{s}e_{s}\right)=f_{w}\left(e_{s}\right)e_{s}=Ae_{s}=Re_{s}$
so that we can write $f_{w}\left(e_{s}\right)=ue_{s}$ for some $u\in R.$
We can assume that $u$ has the form $u=\left(1+rx\right)x^{t}$ for
some $w$ so that $f_{w}\left(e_{s}\right)=\left(1+rx\right)x^{t}e_{s}$.
Thus
\begin{align*}
P_{w} & =f_{w}\left(e_{s}A\right)=f_{w}\left(e_{s}\right)A=\left(1+rx\right)x^{t}e_{s}A=\left(1+rx\right)x^{t}e_{s}R\\
 & =\left(1+rx\right)e_{\omega^{-t}\left(s\right)}x^{t}R=\left(1+rx\right)N_{\omega^{-t}\left(s\right),t}.
\end{align*}
Via $f_{w}$ we have that $P_{w}\cong e_{s}A=e_{s}R=N_{s,0}$. Thus
\begin{align*}
N & =\mathrm{dim}_{\Bbbk}\left(N_{s,0}\right)=\mathrm{dim}_{\Bbbk}\left(P_{w}\right)\\
 & =\mathrm{dim}_{\Bbbk}\left(\left(1+rx\right)N_{\omega^{-t}\left(s\right),t}\right)=\mathrm{dim}_{\Bbbk}\left(N_{\omega^{-t}\left(s\right),t}\right)=N-t
\end{align*}
so that $t=0$ and hence $P_{w}=\left(1+rx\right)N_{s,0}$ and $f_{w}\left(e_{s}\right)=\left(1+rx\right)e_{s}$.
Hence $f_{w}\left(a\right)=\left(1+rx\right)a$ for every $a\in e_{s}A$.

The isomorphism $f_{w}$ induces a bijection
\[
\mathcal{L}_{\mathrm{in}}\left(e_{s}A_{A}\right)\rightarrow\mathcal{L}_{\mathrm{in}}\left(\left(P_{w}\right)_{A}\right):I\mapsto f_{w}\left(I\right)=\left(1+rx\right)I.
\]
Since $M$ is indecomposable and $M\subseteq P_{w}$ then $M\in\mathcal{L}_{\mathrm{in}}\left(\left(P_{w}\right)_{A}\right)$
so that its comes out to be isomorphic to an element of $\mathcal{L}_{\mathrm{in}}\left(e_{s}A_{A}\right)=\left\{ N_{s,t}\mid s\in S,0\leq t\leq N-1\right\} .$
This proves that the $N_{s,t}$'s form an irredundant set of representatives
of $\mathcal{L}_{\mathrm{in}}\left(A_{A}\right).$ Moreover $M=\left(1+rx\right)N_{s,t}$
for some $s\in S,0\leq t\leq N-1$.

We have so proved that $\mathcal{L}_{\mathrm{in}}\left(A_{A}\right)\subseteq\left\{ \left(1+rx\right)N_{s,t}\mid r\in R,s\in S,0\leq t\leq N-1\right\} .$

For the other inclusion, since $\left(1+rx\right)N_{s,t}\cong N_{s,t}$
and $N_{s,t}\in\mathcal{L}_{\mathrm{in}}\left(e_{s}A_{A}\right)$,
we get that $\left(1+rx\right)N_{s,t}$ is indecomposable.
\end{proof}
In the rest of this section we fix a permutation $\mu$ on $S$, a
permutation $\nu$ on $\left\{ 0,1,\dots,N-1\right\} $ and define
a bilinear form $\left\langle -,-\right\rangle :A\times A\to\Bbbk$
by setting, for $s,t\in S,0\leq m,n\leq N-1$,
\begin{equation}
\left\langle e_{s}x^{m},e_{t}x^{n}\right\rangle :=d_{\left(s,m\right)}\delta_{\mu\left(s\right),t}\delta_{\nu\left(m\right),n}\label{eq:formcan}
\end{equation}
where $d_{\left(s,m\right)}\in\Bbbk\setminus\left\{ 0\right\} $ for
all $s,m$.
\begin{lem}
\label{lem:scalarcan}The form $\left\langle -,-\right\rangle $ is
monomial with respect to the basis
\[
\mathcal{B}=\left\{ e_{s}x^{m}\mid s\in S,0\leq m\leq N-1\right\} .
\]
Moreover the Nakayama isomorphism $\gamma$ is given by
\[
\gamma\left(e_{s}x^{m}\right)=\frac{d_{\left(s,m\right)}}{d_{\left(\mu\left(s\right),\nu\left(m\right)\right)}}e_{\mu^{2}\left(s\right)}x^{\nu^{2}\left(m\right)}\qquad\forall s,m.
\]
\end{lem}

\begin{proof}
Set $v_{\left(s,m\right)}:=e_{s}x^{m}$. Then $\left\langle v_{\left(s,m\right)},v_{\left(t,n\right)}\right\rangle =d_{\left(s,m\right)}\delta_{\sigma\left(\left(s,m\right)\right),\left(t,n\right)}$
where $\sigma$ is the permutation defined by $\sigma:=\mu\times\nu:\left(s,m\right)\mapsto\left(\mu\left(s\right),\nu\left(m\right)\right)$.

Thus, by definition, the form $\left\langle -,-\right\rangle $ is
monomial with respect to $\mathcal{B}$. By Lemma \ref{lem:orthog},
there is a right monomial transformation $\gamma=T\left(\tau,c_{\left(s,m\right)}\right)$
(which is necessarily the Nakayama isomorphism) such that \eqref{eq:gamma}
holds true for $\tau:=\sigma^{2}=\mu^{2}\times\nu^{2}$ and
\begin{align*}
c_{\left(s,m\right)} & =\frac{d_{\left(s,m\right)}}{d_{\sigma\left(\left(s,m\right)\right)}}=\frac{d_{\left(s,m\right)}}{d_{\left(\mu\left(s\right),\nu\left(m\right)\right)}}.
\end{align*}
Thus $\gamma\left(v_{\left(s,m\right)}\right)=c_{\left(s,m\right)}v_{\tau\left(\left(s,m\right)\right)}$
and hence $\gamma\left(e_{s}x^{m}\right)=\frac{d_{\left(s,m\right)}}{d_{\left(\mu\left(s\right),\nu\left(m\right)\right)}}e_{\mu^{2}\left(s\right)}x^{\nu^{2}\left(m\right)}$.
\end{proof}
Apply Proposition \ref{prop:action} to $V=A$ and $\left\langle -,-\right\rangle $
as above. Then $A$ has a unique right $A$-module structure $\vartriangleleft$
defined by $\left\langle x\vartriangleleft c,y\right\rangle =\left\langle x,cy\right\rangle $,
for every $x,y,c\in A$ and a unique right $A$-modules structure
$\blacktriangleleft$ defined by $\left\langle x,y\blacktriangleleft c\right\rangle =\left\langle cx,y\right\rangle $.

Moreover, if $a\in A$ is a unit, for every subspace $W$ of $A$
one has
\begin{align*}
\left(aW\right)^{\perp_{L}} & =W^{\perp_{L}}\vartriangleleft a^{-1}\qquad\text{and}\qquad\left(aW\right)^{\perp_{R}}=W^{\perp_{R}}\blacktriangleleft a^{-1}.
\end{align*}
In view of Theorem \ref{thm:indecomposableGold}, we know that
\[
\mathcal{L}_{\mathrm{in}}\left(A_{A}\right)=\left\{ \left(1+rx\right)N_{s,t}\mid r\in R,s\in S,0\leq t\leq N-1\right\}
\]
so that we can compute the orthogonals of all indecomposable right
ideals as
\begin{align}
\left(\left(1+rx\right)N_{s,t}\right)^{\perp_{L}} & =N_{s,t}^{\perp_{L}}\vartriangleleft\left(1+rx\right)^{-1}\label{eq:orthNleft}\\
\left(\left(1+rx\right)N_{s,t}\right)^{\perp_{R}} & =N_{s,t}^{\perp_{R}}\blacktriangleleft\left(1+rx\right)^{-1}.\label{eq:orthNright}
\end{align}
It remains to compute $N_{s,m}^{\perp_{L}}$ and $N_{s,m}^{\perp_{R}}$.
\begin{thm}
\label{thm:orthog-indec}Let $s\in S$ and $m$ be an integer such
that $0\le m\le N-1$. We have
\begin{align*}
N_{s,0}^{\perp_{R}} & =N_{\mu^{2}\left(s\right),0}^{\perp_{L}}=\bigoplus_{t\neq\mu\left(s\right)}N_{t,0}.
\end{align*}
 Assume $\nu\left(m\right):=N-1-m$ for all $m$. Then
\begin{align*}
N_{s,m}^{\perp_{R}} & =N_{\mu^{2}\left(s\right),m}^{\perp_{L}}\qquad\text{and}\qquad N_{s,m}^{\perp_{R}}=N_{s,0}^{\perp_{R}}\oplus N_{\mu\left(s\right),N-m}.
\end{align*}
\end{thm}

\begin{proof}
Note that $N_{s,0}:=e_{s}R$ has basis $B=\left\{ e_{s}x^{m}\mid0\leq m\leq N-1\right\} $
which is a subset of the basis $\mathcal{B}=\left\{ e_{s}x^{m}\mid s\in S,0\leq m\leq N-1\right\} $.
In view of Lemma \ref{lem:scalarcan}, this basis satisfies the conditions
of Lemma \ref{lem:orthog} for $\tau=\mu^{2}\times\nu^{2}$, so that
\begin{align*}
\tau\left(B\right) & =\left\{ e_{\mu^{2}\left(s\right)}x^{\nu^{2}\left(m\right)}\mid0\leq m\leq N-1\right\} =\left\{ e_{\mu^{2}\left(s\right)}x^{n}\mid0\leq n\leq N-1\right\}
\end{align*}
 and hence $N_{s,0}^{\perp_{R}}=W_{B}^{\perp_{R}}=W_{\tau\left(B\right)}^{\perp_{L}}=N_{\mu^{2}\left(s\right),0}^{\perp_{L}}$.

By definition, we have
\[
\left\langle e_{s}x^{m},e_{t}x^{n}\right\rangle :=d_{\left(s,m\right)}\delta_{\mu\left(s\right),t}\delta_{\nu\left(m\right),n}
\]
so if $t\neq\mu\left(s\right)$ then $N_{t,0}\subseteq N_{s,0}^{\perp_{R}}$
and hence $\bigoplus_{t\neq\mu\left(s\right)}N_{t,0}\subseteq N_{s,0}^{\perp_{R}}$.
Now, counting dimensions we get
\begin{align*}
\dim_{\Bbbk}\left(\bigoplus_{t\neq\mu\left(s\right)}N_{t,0}\right) & =\left(\sum_{t\neq\mu\left(s\right)}N\right)=N\left(\left|S\right|-1\right)=N\left|S\right|-N\\
 & =N\left|S\right|-\dim_{\Bbbk}N_{s,0}=\dim_{\Bbbk}N_{s,0}^{\perp_{R}}
\end{align*}
which implies the equality whence $N_{s,0}^{\perp_{R}}=\bigoplus_{t\neq\mu\left(s\right)}N_{t,0}$.

Assume $\nu\left(m\right):=N-1-m$ for all $m$ (note that $\nu^{2}=\mathrm{Id}$).
Since $N_{s,m}:=e_{s}x^{m}R$ has basis $B=\left\{ e_{s}x^{t}\mid m\leq t\leq N-1\right\} $,
we get
\begin{align*}
\tau\left(B\right) & =\left\{ e_{\mu^{2}\left(s\right)}x^{\nu^{2}\left(t\right)}\mid m\leq t\leq N-1\right\} =\left\{ e_{\mu^{2}\left(s\right)}x^{t}\mid m\leq t\leq N-1\right\}
\end{align*}
and hence $N_{s,m}^{\perp_{R}}=N_{\mu^{2}\left(s\right),m}^{\perp_{L}}$.

We now compute $N_{s,m}^{\perp_{R}}$.

If $t=\mu\left(s\right)$ then if $n\ge\nu\left(m\right)+1=N-m$ we
get $\left\langle e_{t}x^{n},e_{s}x^{m}\right\rangle =0$, so $N_{\mu\left(s\right),N-m}\subseteq N_{s,m}^{\perp_{R}}$.
Since $N_{s,0}^{\perp_{R}}=\bigoplus_{t\neq\mu\left(s\right)}N_{t,0}$
we have that $N_{s,0}^{\perp_{R}}+N_{\mu\left(s\right),N-m}=N_{s,0}^{\perp_{R}}\oplus N_{\mu\left(s\right),N-m}$
and hence
\[
N_{s,0}^{\perp_{R}}\oplus N_{\mu\left(s\right),N-m}\subseteq N_{s,m}^{\perp_{R}}.
\]
Now, counting dimensions we get
\begin{align*}
\dim_{\Bbbk}\left(N_{s,0}^{\perp_{R}}\oplus N_{\mu\left(s\right),N-m}\right) & =\left(N\left|S\right|-N\right)+m=N\left|S\right|-N+m\\
 & =N\left|S\right|-\dim_{\Bbbk}N_{s,m}=\dim_{\Bbbk}N_{s,m}^{\perp_{R}}
\end{align*}
which implies the equality whence $N_{s,m}^{\perp_{R}}=N_{s,0}^{\perp_{R}}\oplus N_{\mu\left(s\right),N-m}$
as desired.
\end{proof}

\section{\label{sec:rat}Ideals and their orthogonals for Hopf algebras\label{sec:Hopf-algebra-codes}}

From now on, $H$\textbf{ will always be a finite-dimensional Hopf
algebra} with antipode $S$.

Consider $H^{\ast}$ as a right Hopf module with action $\leftharpoondown$
defined by $\left(f\leftharpoondown h\right)\left(l\right):=f\left(lS\left(h\right)\right)$.
Then, by \cite[Corollary 5.1.6 and Theorem 5.1.3]{Sweedler-Hopf},
we have that $S$ is invertible, the space $\int_{l}(H^{\ast})$ of
left integrals in $H^{\ast}$ is one-dimensional (recall that $\lambda\in H^{*}$
belongs to $\int_{l}(H^{\ast})$ if and only if $\sum h_{1}\lambda(h_{2})=1_{H}\lambda(h)$
for all $h\in H$) and the map
\[
\int_{l}(H^{\ast})\otimes H\rightarrow H^{\ast}:f\otimes h\mapsto\left(f\leftharpoondown h\right)
\]
is bijective. Since $\int_{l}(H^{\ast})$ is one-dimensional, we can
choose a non-zero integral $\lambda\in H^{\ast}$ and the map
\[
\phi:H\rightarrow H^{\ast}:h\mapsto\left(\lambda\leftharpoondown h\right)
\]
is invertible, see \cite[page 306]{DNR}. The map $\phi$ is right
$H$-linear:
\[
\phi\left(hh^{\prime}\right)=\lambda\leftharpoondown\left(hh^{\prime}\right)=\left(\lambda\leftharpoondown h\right)\leftharpoondown h^{\prime}=\phi\left(h\right)\leftharpoondown h^{\prime}.
\]
Note that the structure of right $H$-module on $H^{\ast}$ is not
the canonical one $\leftharpoonup$ given by $\left(f\leftharpoonup h\right)\left(l\right)=f\left(lh\right)$.

Through $\phi$ one has a bijective correspondence between right ideals
in $H$ and right $H$-submodules of $H^{\ast}$.
\begin{rem}
\label{rem:phimenuno}Set $t:=\phi^{-1}\left(\varepsilon\right)\in H$.
By \cite[Section 7.4]{DNR}, we have that $t$ is a right integral
in $H$ (i.e. $th=t\varepsilon\left(h\right)$ for every $h\in H$)
and $\lambda\left(t\right)=1$. Moreover
\[
\phi^{-1}\left(f\right)=\sum t_{1}f\left(t_{2}\right),\text{ for every \ensuremath{f\in H^{*}}}.
\]
Note that given any right integral $t'$ in $H$ such that $\lambda\left(t'\right)=1$
then $S\left(t'\right)$ is a left integral (\cite[Exercise 10.5.1, page 305]{Radford-Hopfbook})
and $\lambda S\left(t'\right)=\lambda\left(t'\right)$ (by the left-right
handed version of \cite[Exercise 10.5.3, page 311]{Radford-Hopfbook})
so that $\phi\left(t'\right)\left(y\right)=\lambda\left(yS\left(t'\right)\right)=\lambda\left(\varepsilon\left(y\right)S\left(t'\right)\right)=\varepsilon\left(y\right)\lambda\left(S\left(t'\right)\right)=\varepsilon\left(y\right)\lambda\left(t'\right)=\varepsilon\left(y\right)$
and hence $\phi\left(t'\right)=\varepsilon$. As a consequence $t'=\phi^{-1}\left(\varepsilon\right)=t.$
\end{rem}

We now are going to apply the results of Section \ref{sec:Bilinear-forms}
in the case $V$ is a Hopf algebra $H.$ In the setting of Section
\ref{sec:rat}, for every $x,y\in H$, we set
\begin{equation}
\left\langle x,y\right\rangle :=\phi\left(x\right)\left(y\right)=\left(\lambda\leftharpoondown x\right)\left(y\right)=\lambda\left[yS\left(x\right)\right].\label{eq:formHopf}
\end{equation}

Note that, since $\phi$ is an isomorphism and $H$ is finite-dimensional,
the bilinear form $\left\langle -,-\right\rangle $ is non-degenerate.

For $x,y,h\in H$, we have
\begin{equation}
\left\langle xh,y\right\rangle =\lambda\left[yS\left(xh\right)\right]=\lambda\left[yS\left(h\right)S\left(x\right)\right]=\left\langle x,yS\left(h\right)\right\rangle \label{eq:Balanced}
\end{equation}

and
\begin{equation}
\sum\left\langle x_{1},y\right\rangle x_{2}=\sum\lambda\left[yS\left(x_{1}\right)\right]x_{2}\overset{(*)}{=}\sum y_{1}\lambda\left[y_{2}S\left(x\right)\right]=\sum y_{1}\left\langle x,y_{2}\right\rangle \label{eq:Sposto}
\end{equation}
where in $(*)$ we used \cite[Lemma 5.1.4]{DNR}.\medskip{}

The Frobenius bilinear form $b:H\times H\to\Bbbk$ is defined by setting
$b\left(x,y\right):=\left(\phi\circ S^{-1}\right)\left(y\right)\left(x\right)=\lambda\left(xy\right)$
for every $x,y\in H$. This bilinear form is non-degenerate because
$\phi\circ S^{-1}$ is invertible. Consider the associated Nakayama
automorphism $\eta:H\to H$ (we know that it is an algebra map because
$H$ is an algebra and the form is associative). Note that $\eta$
is uniquely determined by the equality $b\left(x,y\right)=b\left(y,\eta\left(x\right)\right)$
i.e. $\lambda\left(xy\right)=\lambda\left(y\eta\left(x\right)\right)$
for every $x,y\in H$.
\begin{lem}
\label{lem:leftriang}In the setting of Proposition \ref{prop:action},
take $V$ the underlying algebra of $H$ with left and right regular
actions. Then, for every $h,x\in H$, we have
\begin{align*}
h\vartriangleright x=xS^{-1}\left(h\right)\qquad & \text{and}\qquad h\blacktriangleright x=xS\left(h\right),\\
x\vartriangleleft h=S^{-1}\left(\eta\left(h\right)\right)x\qquad & \text{and}\qquad x\blacktriangleleft h=\eta^{-1}\left(S\left(h\right)\right)x.
\end{align*}
Moreover $\mathrm{Ann}_{_{\bullet\negmedspace}H}\left(h\right)=\left(\eta^{-1}\left(S\left(h\right)\right)\right)^{\perpl}$
and $\mathrm{Ann}_{H_{\bullet}}\left(h\right)=\left(S(h)\right)^{\perpl}$
for every $h\in H$.
\end{lem}

\begin{proof}
By definition of $\vartriangleright,\blacktriangleright$ and \eqref{eq:Balanced},
we have
\[
\left\langle h\vartriangleright x,y\right\rangle =\left\langle x,yh\right\rangle =\left\langle xS^{-1}\left(h\right),y\right\rangle \quad\text{and}\quad\left\langle y,h\blacktriangleright x\right\rangle =\left\langle yh,x\right\rangle =\left\langle y,xS\left(h\right)\right\rangle .
\]
Since the form is non-degenerate, we get $h\vartriangleright x=xS^{-1}\left(h\right)$
and $h\blacktriangleright x=xS\left(h\right)$.

For every $x,y,h\in H$,

\begin{align*}
\left\langle x\vartriangleleft h,y\right\rangle  & =\left\langle x,hy\right\rangle =\lambda\left(hyS\left(x\right)\right)=\lambda\left(yS\left(x\right)\eta\left(h\right)\right)=\left\langle S^{-1}\left(\eta\left(h\right)\right)x,y\right\rangle \\
\left\langle y,x\blacktriangleleft h\right\rangle  & =\left\langle hy,x\right\rangle =\lambda\left(xS\left(y\right)S\left(h\right)\right)=\lambda\left(\eta^{-1}\left(S\left(h\right)\right)xS\left(y\right)\right)=\left\langle y,\eta^{-1}\left(S\left(h\right)\right)x\right\rangle
\end{align*}
so that we obtain $x\vartriangleleft h=S^{-1}\left(\eta\left(h\right)\right)x$
and $x\blacktriangleleft h=\eta^{-1}\left(S\left(h\right)\right)x$.

By Proposition \ref{prop:action}, we have

\begin{align*}
\left(\eta^{-1}\left(S\left(h\right)\right)\right)^{\perpl} & =\left(H\blacktriangleleft h\right)^{\perpl}=\mathrm{Ann}_{_{\bullet\negmedspace}H}\left(h\right)\\
\left(S(h)\right)^{\perpl} & =\left(h\blacktriangleright H\right)^{\perpl}=\mathrm{Ann}_{H_{\bullet}}\left(h\right).
\end{align*}
\end{proof}
\begin{lem}
If $I$ is a right ideal of $H$, then $I^{\perp_{R}}$ and $I^{\perp_{L}}$
are right ideals.
\end{lem}

\begin{proof}
Given $x\in I^{\perp_{L}}$, $y\in I$ and $h\in H$, we get $\left\langle xh,y\right\rangle =\left\langle x,yS\left(h\right)\right\rangle =0$,
as $yS\left(h\right)\in I$, and hence $xh\in I^{\perp_{L}}$, so
$I^{\perp_{L}}$ is a right ideal.

Let now $x\in I^{\perp_{R}}$, $y\in I$ and $h\in H$. Since $H$
is finite-dimensional, $S$ is surjective so there exists an $h^{\prime}\in H$
such that $h=S\left(h^{\prime}\right)$; we get $\left\langle y,xh\right\rangle =\left\langle y,xS\left(h^{\prime}\right)\right\rangle =\left\langle yh^{\prime},x\right\rangle =0$
since $yh^{\prime}\in I$ and hence $xh\in I^{\perp_{R}}$, so $I^{\perp_{R}}$
is a right ideal.
\end{proof}
\begin{rem}
Let $H$ be a Hopf algebra with basis $\mathcal{B}=\left\{ v_{1},\dots,v_{n}\right\} $.
Let $\gamma:H\to H$ be a $\Bbbk$-linear isomorphism such that \eqref{eq:gamma}
holds true. Then $\lambda S\left(x\right)=\left\langle x,1\right\rangle =\left\langle 1,\gamma\left(x\right)\right\rangle =\lambda\gamma\left(x\right)$
so that
\begin{equation}
\lambda S=\lambda\gamma.\label{eq:lambdagamma}
\end{equation}
\end{rem}

\begin{lem}
\label{lem:Involutory} Let $H$ be a Hopf algebra with basis $\mathcal{B}=\left\{ v_{1},\dots,v_{n}\right\} $.
\begin{enumerate}
\item Assume there is a right monomial trasformation $\gamma$ such that
\eqref{eq:gamma} holds true as in Remark \ref{rem:gamma}. Then $\mathrm{\gamma}=\gamma\left(1\right)S^{2}$.
\item Assume $H$ is cosemisimple. If, for every $i\in\left\{ 1,\dots,n\right\} $,
$S^{2}\left(v_{i}\right)=c_{i}v_{i}$, then for every $i,j\in\left\{ 1,\dots,n\right\} $,
$\left\langle v_{i},v_{j}\right\rangle =c_{i}\left\langle v_{j},v_{i}\right\rangle $.
Moreover the map $\gamma:=T\left(\mathrm{Id},c_{i}\right)$ of Remark
\ref{rem:gamma} is exactly $S^{2}$.
\end{enumerate}
\end{lem}

\begin{proof}
$\left(1\right)$. We compute, for every $x,y\in H$
\begin{align*}
\left\langle y,\gamma\left(x\right)\right\rangle  & \overset{\eqref{eq:gamma}}{=}\left\langle x,y\right\rangle \overset{\eqref{eq:Balanced}}{=}\left\langle 1,yS\left(x\right)\right\rangle \overset{\eqref{eq:gamma}}{=}\left\langle yS\left(x\right),\gamma\left(1\right)\right\rangle \overset{\eqref{eq:Balanced}}{=}\left\langle y,\gamma\left(1\right)S^{2}\left(x\right)\right\rangle .
\end{align*}
Since $\left\langle -,-\right\rangle $ is a non-degenerate we then
get $\gamma\left(x\right)=\gamma\left(1\right)S^{2}\left(x\right)$
for every $x\in H$.

$\left(2\right).$ Note that, in our setting, $\left\langle x,y\right\rangle =\lambda\left[yS\left(x\right)\right]$
for some non-zero integral $\lambda$. Since $H$ is cosemisimple,
it has a total integral $\Lambda$. Since $\int_{l}(H^{\ast})$ is
one-dimensional, there is $k\in\Bbbk$ such that $\lambda=k\Lambda.$
Thus, from $\Lambda S=\Lambda$, we obtain $\lambda S=k\Lambda S=k\Lambda=\lambda.$
Since $S$ is invertible and $S^{2}\left(v_{i}\right)\in\Bbbk v_{i}$,
there is a non-zero element $c_{i}\in\Bbbk$ such that $S^{2}\left(v_{i}\right)=c_{i}v_{i}$.
Hence, for every $x,y\in H$, we have
\begin{align*}
c_{i}\left\langle v_{j},v_{i}\right\rangle  & =c_{i}\lambda\left[v_{i}S\left(v_{j}\right)\right]=\lambda\left[S^{2}\left(v_{i}\right)S\left(v_{j}\right)\right]\\
 & =\lambda S\left[v_{j}S\left(v_{i}\right)\right]=\lambda\left[v_{j}S\left(v_{i}\right)\right]=\left\langle v_{i},v_{j}\right\rangle .
\end{align*}
By definition $\gamma\left(v_{i}\right)=c_{i}v_{i}=S^{2}\left(v_{i}\right)$.
\end{proof}
As a consequence we recover the following well-known result.
\begin{cor}
Let $H$ be a finite-dimensional Hopf algebra.
\begin{enumerate}
\item If $\left\langle -,-\right\rangle :H\otimes H\rightarrow\Bbbk$ is
symmetric, then $H$ is involutory (i.e. $S^{2}=\mathrm{Id}_{H}$).
\item Assume $H$ is cosemisimple. If $H$ is involutory, then $\left\langle -,-\right\rangle :H\otimes H\rightarrow\Bbbk$
is symmetric.
\end{enumerate}
\end{cor}

\begin{rem}
Let $H$ be a finite-dimensional involutory Hopf algebra (e.g. $H$
is either commutative or cocommutative, see \cite[Corollary 4.2.8]{DNR}).
Then $H$ is semisimple and $H$ cosemisimple if and only if $\mathrm{char}\Bbbk\nmid\mathrm{dim}_{\Bbbk}H$,
see \cite[Corollary 2.6]{LR-FinDim}.
\end{rem}

\textcolor{red}{}%

\begin{lem}
\label{lem:AnnI}Let $H$ be a finite-dimensional Hopf algebra with
basis $\mathcal{B}=\left\{ v_{1},\dots,v_{n}\right\} $ and let $I$
be a right ideal of $H$. Then $S^{-1}\left(\mathrm{Ann}_{H}\left(I\right)\right)\subseteq I^{\perp_{L}}$,
where $S^{-1}$ denotes the composition inverse of the antipode.

Assume there is a right monomial transformation $\gamma$ such that
\eqref{eq:gamma} holds true. If $I^{\perp_{L}}$ is a two-sided ideal,
then equality holds.
\end{lem}

\begin{proof}
Let $x\in\mathrm{Ann}_{H}\left(I\right):=\left\{ h\in H\mid yh=0,\forall y\in I\right\} $.
Then, for every $y\in I$, we have $\left\langle S^{-1}\left(x\right),y\right\rangle \overset{\eqref{eq:Balanced}}{=}\left\langle 1,yx\right\rangle =0$
so that $S^{-1}\left(x\right)\in I^{\perp_{L}}$.

Let us prove the last part of the statement. Let $x\in I^{\perp_{L}}$.
Then for every $z\in H$, we have
\[
\left\langle yS\left(x\right),z\right\rangle \overset{\eqref{eq:gamma}}{=}\left\langle \gamma^{-1}\left(z\right),yS\left(x\right)\right\rangle \overset{\eqref{eq:Balanced}}{=}\left\langle \gamma^{-1}\left(z\right)x,y\right\rangle \overset{\gamma^{-1}\left(z\right)x\in I^{\perp_{L}}}{=}0.
\]
Since $\left\langle -,-\right\rangle $ is a non-degenerate, we get
that $yS\left(x\right)=0$ and hence $S\left(x\right)\in\mathrm{Ann}_{H}\left(I\right)$.
Thus $S\left(I^{\perp_{L}}\right)\subseteq\mathrm{Ann}_{H}\left(I\right)$.
Applying $S^{-1}$ on both sides we obtain $I^{\perp_{L}}\subseteq S^{-1}\left(\mathrm{Ann}_{H}\left(I\right)\right)$.
\end{proof}

\specialsection{Examples and applications\label{sec:Examples-and-applications}}

In the present section, for every $n\geq2$ and $a,b\in\mathbb{Z},$
we use the notation
\[
\delta_{a,b}^{\equiv_{n}}:=\left\{ \begin{array}{cc}
1, & \text{if }a\equiv_{n}b\\
0, & \text{otherwise}
\end{array}\right.
\]

\subsection{Cyclic group algebra \label{subsec:Cyclic-group-algebra}}

The aim of this section is to show how our treatment specializes to
classical cyclic codes in case of the cyclic group algebra $H:=\Bbbk\left[X\right]/\left(X^{n}-1\right)=\Bbbk\left\langle x\mid x^{n}=1\right\rangle $
where $x:=X+\left(X^{n}-1\right)$. \footnote{Note that here we use the notation $x$ for the generator because
it is standard in Code Theory. Later on, dealing with Taft algebras,
we will use the notation $g$ for the same element reminding it is
a group-like element. }

Consider $\lambda\in H^{\ast}$ defined on generators by $\lambda\left(x^{i}\right):=\delta_{i,0}$
for $0\leq i\leq n-1$, see \cite[Example 5.2.9-2)]{DNR}. Now $t=\sum_{i=0}^{n-1}x^{i}$
is both a left and right integral (see \cite[Examples 1)]{Sweedler-Hopf}),
and $\lambda\left(t\right)=1$ so that, by Remark \ref{rem:phimenuno},
\[
\phi^{-1}\left(f\right)=\sum t_{1}f\left(t_{2}\right)=\sum_{i=0}^{n-1}x^{i}f\left(x^{i}\right),\text{ for every \ensuremath{f\in H^{*}}}.
\]
In this specific case we can express $\phi$ explicitly as follows:
\[
\phi:H\rightarrow H^{\ast}:c_{0}+c_{1}x+\cdots+c_{n-1}x^{n-1}\mapsto\sum c_{i}\phi\left(x^{i}\right)
\]
where $\phi\left(x^{i}\right)\in H^{*}$ is defined by $\phi\left(x^{i}\right)(x^{j})=\lambda\left(x^{j-i}\right)=\delta_{i,j},\forall j\in\left\{ 0,\ldots,n-1\right\} $
so that $\phi\left(x^{0}\right),\ldots,\phi\left(x^{n-1}\right)$
is the dual basis of $1,\ldots,x^{n-1}$.

Consider the isomorphism $\alpha\colon H^{*}\rightarrow\Bbbk^{n}:f\mapsto\left(f_{0},f_{1},\cdots,f_{n-1}\right)$
where $f_{i}:=f\left(x^{i}\right)$. It permits to regard the elements
of $H^{\ast}$ as code words, although it strictly depends on the
chosen basis for $H.$ In view of \cite[2.2, page 70]{Abe-Hopf},
$\alpha$ is an algebra map, where $H^{\ast}$ carries the convolution
product. The right $H$-module structure of $H^{\ast}$ given by $\leftharpoondown$
induces, via $\alpha$, a right $H$-module structure on $\Bbbk^{n}$.
We want to give this structure explicitly. To this aim, for every
$f\in H^{*}$ and $0\leq i\leq n-1$, let $f_{i}:=f\left(x^{i}\right)$.
In other words $\alpha\left(f\right)=\left(f_{0},f_{1},\cdots,f_{n-1}\right)$.
We compute
\[
\left(f\leftharpoondown x\right)_{i}=\left(f\leftharpoondown x\right)\left(x^{i}\right)=f\left(x^{i-1}\right)=\begin{cases}
f_{n-1} & \text{if }i=0,\\
f_{i-1} & \text{if }1\leq i\leq n-1
\end{cases}
\]
so that $\alpha\left(f\leftharpoondown x\right)=\left(f_{n-1},f_{0},f_{1},\cdots,f_{n-2}\right).$
Therefore the unique right $H$-module structure on $\Bbbk^{n}$ which
makes $\alpha$ a morphism of right $H$-modules is given by
\begin{equation}
\left(f_{0},f_{1},\cdots,f_{n-1}\right)\leftharpoondown x:=\left(f_{n-1},f_{0},f_{1},\cdots,f_{n-2}\right).\label{eq:actioncyclic}
\end{equation}
Thus we recover the right $H$-linear bijection \eqref{eq: KCn} as
\[
\alpha\circ\phi:\Bbbk\left\langle x\right\rangle \rightarrow\Bbbk^{n}:c_{0}+c_{1}x+\cdots+c_{n-1}x^{n-1}\mapsto\left(c_{0},c_{1},\cdots,c_{n-1}\right).
\]
This map gives a bijective correspondence between right ideals of
$H$ and right $H$-submodules of $\Bbbk^{n}$, with respect to the
action $\leftharpoondown$ . By formula \eqref{eq:actioncyclic},
these submodules are exactly the cyclic codes of length $n$.

Note also that
\begin{equation}
\left\langle x^{i},x^{j}\right\rangle =\lambda\left[x^{j}S\left(x^{i}\right)\right]=\lambda\left(x^{j-i}\right)=\delta_{i,j}\label{eq:scalar}
\end{equation}
so that $\left\langle -,-\right\rangle :H\otimes H\rightarrow\Bbbk$
is symmetric, $\left\{ 1,x,x^{2},\ldots,x^{n-1}\right\} $ is an orthonormal
basis and through \eqref{eq: KCn}, it corresponds to the scalar product
on $\Bbbk^{n}$. Since the bilinear form is symmetric the left and
the right orthogonal coincide.

Let $q\in\Bbbk$ be a primitive $n$-th root of unity.

Note that $\Bbbk\left\langle x\right\rangle $ identifies with the
group algebra $\Bbbk G$ over the cyclic group $G:=\left\langle x\right\rangle $
with $x$ of order $n$, as above. The existence of a primitive $n$-th
root of unity implies $\mathrm{char}\left(\Bbbk\right)\nmid n$ so
that, for every $t\in\mathbb{Z},$ we can consider
\[
e_{t}:=\frac{1}{n}\sum_{i=0}^{n-1}q^{ti}x^{i}.
\]
 It is well-known that the $e_{t}$'s form a complete set of orthogonal
idempotents in $\Bbbk G$. This fact is the main tool used in \cite{CGL}.

As a consequence $\Bbbk\left\langle x\right\rangle $ can be regarded
as the algebra $\Bbbk\left(\omega,N\right)$ where $\omega=\mathrm{Id}$
and $N=1$ with basis $\left\{ e_{t}\mid0\leq t\leq n-1\right\} $.

For $0\leq s,t\leq n-1$, we have
\begin{align*}
\left\langle e_{s},e_{t}\right\rangle  & =\left\langle \frac{1}{n}\sum_{i=0}^{n-1}q^{si}x^{i},\frac{1}{n}\sum_{j=0}^{n-1}q^{tj}x^{j}\right\rangle \overset{\eqref{eq:scalar}}{=}\frac{1}{n^{2}}\sum_{i=0}^{n-1}q^{\left(s+t\right)i}=\frac{1}{n}\delta_{s+t,0}^{\equiv_{n}}.
\end{align*}
In particular the form $\left\langle -,-\right\rangle $ is as in
\eqref{eq:formcan}, where $\mu\left(t\right):=\left[-t\right]_{n},$
$\nu\left(0\right):=0$ and $d_{\left(s,0\right)}:=\frac{1}{n}$.
Here $\left[t\right]_{n}$ denotes the remainder modulo $n$ of $t$.
As a consequence $\left\langle -,-\right\rangle $ is monomial with
respect to $\mathcal{B}=\left\{ e_{s}\mid0\leq s\leq n-1\right\} $.
Since the form is symmetric, we get that the Nakayama isomorphism
$\gamma$ is the identity. Note also that, since $\Bbbk\left\langle x\right\rangle $
is commutative, also the Nakayama automorphism $\eta$ is the identity.

By Lemma \ref{lem:leftriang} and since $S^{-1}=S$, we have $x\vartriangleleft h=S\left(h\right)x=x\blacktriangleleft h$.

By Theorem \ref{thm:indecomposableGold}, we have $\mathcal{L}_{\mathrm{in}}\left(A_{A}\right)=\left\{ N_{s,0}=e_{s}\Bbbk\mid0\leq s\leq n-1\right\} $.

By Theorem \ref{thm:orthog-indec}, we have $N_{s,0}^{\perp}=\bigoplus_{t\neq\mu\left(s\right)}N_{t,0}=\bigoplus_{t\not\equiv_{n}-s}N_{t,0}$.
\begin{rem}
Consider a cyclic code $I$ and its dual $I^{\perp}$. We want to
show that $I^{\perp}$ is, indeed, the classical orthogonal code.
Consider the generator polynomial $g\left(X\right)=g_{0}+g_{1}X+\cdots+g_{s-1}X^{s-1}+X^{s}$
for the code $I$ so that $I=\left(g\left(x\right)\right)$. Note
that $d:=\mathrm{dim}I=n-s$ as a $\Bbbk$-basis for $I$ is given
by $x^{i}g\left(x\right),0\leq i\leq n-s-1$. Let
\[
h\left(X\right):=h_{0}+h_{1}X+\cdots+h_{d-1}X^{d-1}+X^{d}
\]
be the unique monic polynomial such that $g\left(X\right)h\left(X\right)=X^{n}-1$,
i.e. the parity-check polynomial. Remark that $h(0)=h_{0}\neq0$ and
$g(x)h(x)=0$. By Lemma \ref{lem:leftriang} we have
\begin{align*}
\left(S(h(x))\right)^{\perpl} & =\mathrm{Ann}_{H_{\bullet}}\left(h(x)\right)=\left(g\left(x\right)\right)=I.
\end{align*}
Since $\perpl=\perpr$ we get $I^{\perp}=\left(S(h(x))\right)^{\perp\perp}=\left(S(h(x))\right)$.
Now define
\[
g^{\perp}(X)=h_{0}^{-1}X^{d}h(X^{-1})\in\Bbbk\left[X\right].
\]
We have $\deg g^{\perp}(X)=d=\deg h(X)$ and $g^{\perp}(x)=h_{0}^{-1}x^{d}S(h(x))\in I^{\perp}$
thus $\left(g^{\perp}(x)\right)\subseteq I^{\perp}$. Since $\mathrm{dim}\left(g^{\perp}(x)\right)=n-\deg g^{\perp}(X)=n-d=s=\mathrm{dim}\left(I^{\perp}\right)$,
we can conclude
\[
I^{\perp}=\left(S(h(x))\right)=\left(g^{\perp}(x)\right)
\]
so $I^{\perp}$ is the classical dual code.

\end{rem}

\subsection{Taft algebra \label{subsec:Taft-algebra}}

Let us consider the main example investigated in \cite{CGL}. Let
$N\geq2$ be an integer and let $q\in\Bbbk$ be a primitive $N$-th
root of unity. Consider the Taft algebra
\[
A=\Bbbk\left\langle g,x\mid g^{N}=1,x^{N}=0,gx=qxg\right\rangle .
\]
It is a Hopf algebra in a unique way such that
\[
\Delta\left(g\right)=g\otimes g\qquad\text{and}\qquad\Delta\left(x\right)=g\otimes x+x\otimes1.
\]
Note that $A$ has basis $\left\{ x^{n}g^{a}\mid0\leq n,a\leq N-1\right\} .$
From the structure above it follows that $S\left(g\right)=g^{-1}$
and $S\left(x\right)=-g^{-1}x=-q^{-1}xg^{-1}$.

Let $R=\Bbbk\left[x\mid x^{N}=0\right]$ be the subalgebra of $A$
generated by $x$.

Set $G:=\left\langle g\right\rangle $ the group of group-like elements
in $A$ and set $H:=\Bbbk G$. For the reader's sake we include in
the following lemmas the proofs of main facts that will be used later
on.
\begin{lem}
\label{lem:formuTaft}The following equalities hold for every $m\in\mathbb{N},t,a\in\mathbb{\mathbb{Z}},$

\begin{align}
e_{t}x^{m} & =x^{m}e_{t+m},\qquad g^{a}e_{t}=e_{t}g^{a}=q^{-ta}e_{t}\label{eq:TaftId}\\
S\left(e_{t}\right) & =e_{-t},\qquad S\left(x^{m}\right)=\left(-1\right)^{m}q^{\frac{-m\left(m+1\right)}{2}}x^{m}g^{-m},\qquad S^{2}\left(x^{m}\right)=q^{-m}x^{m}.\label{eq:antipode}
\end{align}
A basis of $A$ is given by $\left\{ x^{m}e_{s}\mid0\leq s,m\leq N-1\right\} $.

As a consequence $A=\Bbbk\left(\omega,N\right)$ as in Section \ref{sec:Indec-ideals}
where the permutation $\omega$ of the indexes of the $e_{s}$'s is
defined by $\omega\left(s\right):=s+1$ modulo $N$.
\end{lem}

\begin{proof}
We compute, for every $m\in\mathbb{N},t,a\in\mathbb{\mathbb{Z}},$
\begin{align*}
e_{t}x^{m} & =\frac{1}{N}\sum_{i=0}^{N-1}q^{ti}g^{i}x^{m}=\frac{1}{N}\sum_{i=0}^{N-1}q^{\left(t+m\right)i}x^{m}g^{i}=x^{m}e_{t+m},\\
e_{t}g^{a} & =\frac{1}{N}\sum_{i=0}^{N-1}q^{ti}g^{i+a}=\frac{1}{N}\sum_{i=0}^{N-1}q^{-ta}q^{t\left(i+a\right)}g^{i+a}=q^{-ta}e_{t},\\
S\left(e_{t}\right) & =\frac{1}{N}\sum_{a=0}^{N-1}q^{ta}g^{-a}=\frac{1}{N}\sum_{a=0}^{N-1}q^{\left(-t\right)\left(-a\right)}g^{-a}=e_{-t},\\
S\left(x^{m}\right) & =S\left(x\right)^{m}=\left(-g^{-1}x\right)^{m}=\left(-1\right)^{m}q^{\frac{-m\left(m+1\right)}{2}}x^{m}g^{-m},\\
S^{2}\left(x^{m}\right) & =\left(S\left(-g^{-1}x\right)\right)^{m}=\left(-S\left(x\right)S\left(g^{-1}\right)\right)^{m}=\left(g^{-1}xg\right)^{m}=q^{-m}x^{m}.
\end{align*}
Since $e_{t}g^{a}=g^{a}e_{t},$ we get the thesis. These equations
guarantee that the set $\left\{ x^{m}e_{s}\mid0\leq s,m\leq N-1\right\} $
is a basis of $A$.
\end{proof}
\begin{lem}
We have $\int_{l}\left(A^{\ast}\right)=\Bbbk\lambda$ where
\begin{align}
\lambda\left(x^{m}g^{a}\right) & =\delta_{m,N-1}\delta_{a,1}^{\equiv_{N}}\text{ for all }m\in\mathbb{N},a\in\mathbb{Z},\nonumber \\
\lambda\left(x^{m}e_{s}\right) & =\frac{1}{N}q^{s}\delta_{m,N-1}\text{ for all }m\in\mathbb{N},s\in\mathbb{Z}.\label{eq:lambdaT}
\end{align}
\end{lem}

\begin{proof}
One checks that $\int_{l}\left(H^{\ast}\right)=\Bbbk\lambda$ where
$\lambda\left(x^{m}g^{a}\right)=\delta_{m,N-1}\delta_{a,1}$ for $0\leq m,a\leq N-1.$
From this expression for $\lambda$ one easily deduces the general
one for $m\in\mathbb{N},a\in\mathbb{Z}$. From it we get
\begin{align*}
\lambda\left(x^{m}e_{s}\right) & =\lambda\left(x^{m}\frac{1}{N}\sum_{i=0}^{N-1}q^{si}g^{i}\right)=\frac{1}{N}\sum_{i=0}^{N-1}q^{si}\lambda\left(x^{m}g^{i}\right)\\
 & =\frac{1}{N}\sum_{i=0}^{N-1}q^{si}\delta_{m,N-1}\delta_{i,1}=\frac{1}{N}q^{s}\delta_{m,N-1}.
\end{align*}
\end{proof}
We now compute explicitly our bilinear form in two slightly different
basis, the first one needed in the proof of Theorem \ref{thm:orthogonal}.
\begin{lem}
\label{lem:scalar-1}We have
\begin{align*}
\left\langle e_{t}x^{n},e_{s}x^{m}\right\rangle  & =\frac{1}{N}\left(-1\right)^{n}q^{\frac{-\left(n+2t\right)\left(n+1\right)}{2}}\delta_{s+t,1}^{\equiv N}\delta_{m+n,N-1};\\
\left\langle x^{v}e_{b},x^{u}e_{a}\right\rangle  & =\frac{1}{N}\left(-1\right)^{v}q^{\frac{\left(v-2b\right)\left(v+1\right)}{2}}\delta_{a+b,0}^{\equiv N}\delta_{u+v,N-1}.
\end{align*}

In particular the form $\left\langle -,-\right\rangle $ is as in
\eqref{eq:formcan}, where $\mu\left(t\right):=\left[1-t\right]_{N}$,
$\nu\left(n\right):=N-1-n$ and $d_{\left(s,m\right)}:=\frac{1}{N}\left(-1\right)^{m}q^{\frac{-\left(m+2s\right)\left(m+1\right)}{2}}$.
Here $\left[t\right]_{N}$ denotes the remainder modulo $N$ of $t$.
As a consequence $\left\langle -,-\right\rangle $ is monomial with
respect to $\mathcal{B}=\left\{ e_{s}x^{m}\mid0\leq s,m\leq N-1\right\} $.
Moreover the Nakayama isomorphism $\gamma$ is given by $\gamma\left(h\right)=hg$,
for every $h\in A$.
\end{lem}

\begin{proof}
We compute
\begin{align*}
e_{s}x^{m}S\left(e_{t}x^{n}\right) & =e_{s}x^{m}S\left(x^{n}\right)S\left(e_{t}\right)\\
 & =e_{s}x^{m}\left(-1\right)^{n}q^{\frac{-n\left(n+1\right)}{2}}x^{n}g^{-n}e_{-t}\\
 & =\left(-1\right)^{n}q^{\frac{-n\left(n+1\right)}{2}}e_{s}x^{m+n}g^{-n}e_{-t}\\
 & =\left(-1\right)^{n}q^{\frac{-n\left(n+1\right)}{2}-tn}e_{s}x^{m+n}e_{-t}\\
 & =\left(-1\right)^{n}q^{\frac{-n\left(n+1\right)}{2}-tn}x^{m+n}e_{s+m+n}e_{-t}\\
 & =\left(-1\right)^{n}q^{\frac{-n\left(n+1\right)}{2}-tn}\delta_{s+t+m+n,0}^{\equiv N}x^{m+n}e_{-t}.
\end{align*}
Thus we have
\begin{align*}
\left\langle e_{t}x^{n},e_{s}x^{m}\right\rangle  & =\lambda\left[e_{s}x^{m}S\left(e_{t}x^{n}\right)\right]\\
 & =\left(-1\right)^{n}q^{\frac{-n\left(n+1\right)}{2}-tn}\delta_{s+t+m+n,0}^{\equiv N}\lambda\left(x^{m+n}e_{-t}\right)\\
 & =\left(-1\right)^{n}q^{\frac{-n\left(n+1\right)}{2}-tn}\frac{1}{N}q^{-t}\delta_{s+t+m+n,0}^{\equiv N}\delta_{m+n,N-1}\\
 & =\frac{1}{N}\left(-1\right)^{n}q^{\frac{-\left(n+2t\right)\left(n+1\right)}{2}}\delta_{s+t+N-1,0}^{\equiv N}\delta_{m+n,N-1}\\
 & =\frac{1}{N}\left(-1\right)^{n}q^{\frac{-\left(n+2t\right)\left(n+1\right)}{2}}\delta_{s+t,1}^{\equiv N}\delta_{m+n,N-1}.
\end{align*}
Moreover
\begin{align*}
\left\langle x^{v}e_{b},x^{u}e_{a}\right\rangle  & =\left\langle e_{b-v}x^{v},e_{a-u}x^{u}\right\rangle =\frac{1}{N}\left(-1\right)^{v}q^{\frac{-\left(v+2b-2v\right)\left(v+1\right)}{2}}\delta_{a-u+b-v,1}^{\equiv N}\delta_{u+v,N-1}\\
 & =\frac{1}{N}\left(-1\right)^{v}q^{\frac{\left(v-2b\right)\left(v+1\right)}{2}}\delta_{a+b,0}^{\equiv N}\delta_{u+v,N-1}.
\end{align*}

By the foregoing, it is clear that $\left\langle -,-\right\rangle $
is as in \ref{eq:formcan}, where $\mu\left(t\right):=\left[1-t\right]_{N}$
and $\nu\left(n\right):=N-1-n$.

Thus, by Lemma \ref{lem:scalarcan} the form $\left\langle -,-\right\rangle $
is monomial with respect to the basis $\mathcal{B}=\left\{ e_{s}x^{m}\mid0\leq s,m\leq N-1\right\} $.
Moreover, $\gamma$ is given by
\[
\gamma\left(e_{s}x^{m}\right)=\frac{d_{\left(s,m\right)}}{d_{\left(\mu\left(s\right),\nu\left(m\right)\right)}}e_{\mu^{2}\left(s\right)}x^{\nu^{2}\left(m\right)}\qquad\forall s,m.
\]
 Since $\mu^{2}=\mathrm{Id}=\nu^{2}$ we obtain $\gamma\left(e_{s}x^{m}\right)=\frac{d_{\left(s,m\right)}}{d_{\left(\mu\left(s\right),\nu\left(m\right)\right)}}e_{s}x^{m}$.
We compute
\begin{align*}
\frac{d_{\left(s,m\right)}}{d_{\left(\mu\left(s\right),\nu\left(m\right)\right)}} & =\frac{\frac{1}{N}\left(-1\right)^{m}q^{\frac{-\left(m+2s\right)\left(m+1\right)}{2}}}{\frac{1}{N}\left(-1\right)^{N-1-m}q^{\frac{-\left(N-1-m+2-2s\right)\left(N-1-m+1\right)}{2}}}\\
 & =\left(-1\right)^{N-1}q^{\frac{-\left(m+2s\right)\left(m+1\right)+\left(N-m+1-2s\right)\left(N-m\right)}{2}}\\
 & =\left(-1\right)^{N-1}q^{\frac{N-2m-2s-2Nm-2Ns+N\text{\texttwosuperior}}{2}}=\left(-1\right)^{N-1}q^{-\frac{N(N+1)}{2}-\left(m+s\right)}\\
 & =\left(-1\right)^{N-1}\left(-1\right)^{N-1}q^{-\left(m+s\right)}=q^{-\left(m+s\right)}.
\end{align*}
where we note that $x^{N}-1=\prod_{i=0}^{N-1}\left(x-q^{i}\right)$
implies $-1=\left(-1\right)^{N}q^{\frac{N\left(N-1\right)}{2}}$ and
hence $q^{-\frac{N(N+1)}{2}}=\left(-1\right)^{N-1}$. Thus $\gamma\left(e_{s}x^{m}\right)=q^{-\left(m+s\right)}e_{s}x^{m}=q^{-m}e_{s}gx^{m}=e_{s}x^{m}g$
and hence $\gamma\left(h\right)=hg$ for every $h\in H$.
\end{proof}
\begin{lem}
\label{lem:lambdaS}In the setting of Proposition \ref{prop:action},
take $V$ the Taft algebra $A$ with left regular action. Then $\vartriangleleft=\blacktriangleleft$
and, for every $x\in H,r\in R$, we have $x\vartriangleleft r=S\left(r\right)x.$
\end{lem}

\begin{proof}
By Lemma \ref{lem:scalar-1}, we have that the Nakayama isomorphism
$\gamma$ is given by $\gamma\left(h\right)=hg,$ for every $h\in A$.
Thus $\gamma$ is left $A$-linear. By Proposition \ref{prop:action},
we have that $\vartriangleleft=\blacktriangleleft$.

By \eqref{eq:lambdagamma}, we have that
\[
\lambda S\left(x^{n}e_{t}\right)=\lambda\gamma\left(x^{n}e_{t}\right)=\lambda\left(x^{n}e_{t}g\right)\overset{\eqref{eq:TaftId}}{=}q^{-t}\lambda\left(x^{n}e_{t}\right)\overset{\eqref{eq:lambdaT}}{=}=\frac{1}{N}\delta_{n,N-1}.
\]

We compute
\begin{align*}
\lambda\left(x^{u}e_{s}x^{m}\right) & =\lambda\left(e_{s-u}x^{u}x^{m}\right)\\
 & =\lambda\left(e_{s-u}x^{m+u}\right)\\
 & =\frac{1}{N}q^{s+m}\delta_{m+u,N-1}\\
 & =q^{-u}\frac{1}{N}q^{s+m+u}\delta_{m+u,N-1}\\
 & =q^{-u}\lambda\left(e_{s}x^{m+u}\right)\\
 & =\lambda\left(e_{s}x^{m}q^{-u}x^{u}\right)\\
 & \overset{\eqref{eq:antipode}}{=}\lambda\left(e_{s}x^{m}S^{2}\left(x^{u}\right)\right)
\end{align*}
which implies $\lambda\left(rh\right)=\lambda\left(hS^{2}\left(r\right)\right)$
for every $r\in R,h\in A$. As a consequence, if $\eta$ is the Nakayama
automorphism, we get $\eta\left(r\right)=S^{2}\left(r\right)$ for
every $r\in R$. By Lemma \ref{lem:leftriang}we have $x\vartriangleleft r=S^{-1}\left(\eta\left(r\right)\right)x=S\left(r\right)x.$
\end{proof}
We are now ready to compute the indecomposable ideals and their orthogonals.

Denote by $\mathcal{L}\left(M_{R}\right)$ the set of right $R$-submodules
of a given right $R$-module $M_{R}.$ Let $\mathcal{L}_{\mathrm{in}}\left(M_{R}\right)$
denote the set of right $R$-submodules which are indecomposable.
\begin{thm}
\label{thm:indecomposable}Consider the Taft Hopf algebra
\[
A=\Bbbk\left\langle g,x\mid g^{N}=1,x^{N}=0,gx=qxg\right\rangle
\]
 and let $R$ be the subalgebra of $A$ generated by $x.$ As in \cite{CGL},
for $s,t=0,\ldots,N-1$ set $N_{s,t}:=e_{s}J^{t}=e_{s}x^{t}R.$ Then
the $N_{s,t}$'s form an irredundant set of representatives of $\mathcal{L}_{\mathrm{in}}\left(A_{A}\right)$
and
\[
\mathcal{L}_{\mathrm{in}}\left(A_{A}\right)=\left\{ \left(1+rx\right)N_{s,t}\mid r\in R,0\leq s,t\leq N-1\right\} .
\]
\end{thm}

\begin{proof}
By Lemma \ref{lem:formuTaft} $A$ is of the form $\Bbbk\left(\omega,N\right)$
as in Section \ref{sec:Indec-ideals}. Thus the statement follows
by Theorem \ref{thm:indecomposableGold}.
\end{proof}
\begin{thm}
\label{thm:orthogonal}Let $a=a\left(x\right)\in R$ be an invertible
element (we can assume $a\left(0\right)=1)$, $s,m$ integers such
that $0\le s,m\le N-1$; then we have
\begin{align*}
N_{s,0}^{\perp} & =\bigoplus_{t\not\equiv_{N}1-s}N_{t,0},\qquad N_{s,m}^{\perp}=N_{s,0}^{\perp}\oplus N_{1-s,N-m}\quad\text{and}\\
 & \left(aN_{s,m}\right)^{\perp}=S\left(a^{-1}\right)N_{s,m}^{\perp}.
\end{align*}
\end{thm}

\begin{proof}
By Lemma \ref{lem:lambdaS} and \eqref{eq:orthNleft},\eqref{eq:orthNright},
we have that $\left(aN_{s,m}\right)^{\perp_{L}}=S\left(a^{-1}\right)N_{s,m}^{\perp_{L}}$
and $\left(aN_{s,m}\right)^{\perp_{R}}=S\left(a^{-1}\right)N_{s,m}^{\perp_{R}}$
for every $s,m$. Note that $\mu\left(s\right)=1-s$ modulo $N$ (so
that $\mu^{2}=\mathrm{Id}$) and $\nu\left(m\right)=N-1-m$ in our
case. Thus by Theorem \ref{thm:orthog-indec}, we get $N_{s,m}^{\perp_{R}}=N_{s,m}^{\perp_{L}}$
for every $s,m$ (hence we can use the notation $\perp$) and the
equalities in the present statement holds.
\end{proof}

\subsection{Another example\label{subsec:Another-example}}

Consider the commutative Hopf algebra $H=\widehat{D}_{n}$ of \cite[Section 4]{CDMM}
for $n=6$, where $D_{n}$ denotes the dihedral group of order $2n$
(note that in \cite{CDMM} it is denoted by $D_{2n}$). Assume that
$\Bbbk$ contains $\zeta$ a primitive $6$-th root of $1$ and observe
that in this case $H\cong\left(\Bbbk D_{n}\right)^{*}\cong\Bbbk^{D_{n}}$
so that $H$ is both semisimple and cosemisimple. Recall that $H$
is given by the generators $a,b$ and relations $a^{2}=1=b^{6}$ and
$ab=ba$. The element $a$ is group-like while
\[
\Delta\left(b\right)=b\otimes e_{0}b+b^{-1}\otimes e_{1}b,\qquad\varepsilon\left(b\right)=1,\qquad S\left(b\right)=e_{0}b^{-1}+e_{1}b,
\]
where
\[
e_{0}=\frac{1}{2}\left(1+a\right)\qquad\text{and}\qquad e_{1}=\frac{1}{2}\left(1-a\right).
\]
We also set
\[
f_{j}=\frac{1}{6}\sum_{i=0}^{5}\zeta^{ji}b^{i}\qquad\text{and}\qquad e_{i,j}:=e_{i}f_{j}.
\]
Since the subalgebra $\Bbbk\left\langle b\right\rangle $ of $H$
generated by $b$ is a group algebra (note it is not a subbialgebra
since is not group-like), as in Subsection \eqref{subsec:Cyclic-group-algebra},
we get that the $f_{j}$'s form a complete set of orthogonal idempotents
in $\Bbbk\left\langle b\right\rangle $. As a consequence, since $H$
is commutative, the $e_{i,j}$'s form a complete set of orthogonal
idempotents in $H$. It is clear that $\left\{ e_{i,j}\mid0\leq i\leq1,0\leq j\leq5\right\} $
is a generating set whence a basis for $H$ over $\Bbbk$.

In this section, consider as $A$ the Hopf algebra $R\#H$ in \cite[Theorem 4.1]{CDMM},
where $R:=R_{q}\left(H,g,\chi\right)$ where $g:=b^{3}$ (note that
$g$ is group-like) and $\chi:H\to\Bbbk$ is defined by $\chi\left(a\right)=1$
and $\chi\left(b\right)=-1$. Note that $q=\chi\left(g\right)=\chi\left(b\right)^{3}=-1$
so that $N=o\left(q\right)=2$.

Following \cite[Section 2]{CDMM}, we get that $A$ is given by the
generators $x,a,b$ such that $H$ is a Hopf subalgebra of $A$ and
with the further relations $x^{2}=0,hx=x\sum\chi\left(h_{1}\right)h_{2}$
for every $h\in H$. Taking $h=a,b$ we get $ax=xa$ and $bx=-xb$
respectively. Moreover
\[
\Delta\left(x\right)=g\otimes x+x\otimes1=b^{3}\otimes x+x\otimes1,\qquad\varepsilon\left(x\right)=0,\qquad S\left(x\right)=-gx=xg.
\]
By the foregoing $\left\{ x^{m}e_{i,j}\mid0\leq m,i\leq1,0\leq j\leq5\right\} $
is a basis for $A$ over $\Bbbk$.

\begin{lem}
\label{lem:formCDMM}The following equality holds for every $m\in\mathbb{N},i,j\in\mathbb{Z}$
\begin{align}
e_{i,j}x^{m} & =x^{m}e_{i,j+3m},\qquad e_{i,j}g^{m}=g^{m}e_{i,j}=\left(-1\right)^{jm}e_{i,j},\label{eq:OtherId1}\\
e_{i,j}a^{m} & =a^{m}e_{i,j}=\left(-1\right)^{im}e_{i,j},\qquad e_{i,j}b^{m}=b^{m}e_{i,j}=\zeta^{-jm}e_{i,j},\label{eq:OtherId2}\\
S\left(e_{i,j}\right) & =\left(e_{0}f_{-j}+e_{1}f_{j}\right)e_{i}=e_{i,\left(-1\right)^{i+1}j},\nonumber \\
S\left(x^{m}\right) & =x^{m}g^{m},\nonumber \\
S^{2}\left(x^{m}\right) & =\left(-1\right)^{m}x^{m}.\nonumber
\end{align}

As a consequence $A=\Bbbk\left(\omega,N\right)$ as in Section \ref{sec:Indec-ideals}
where the permutation $\omega$ of the indexes of the $e_{i,j}$'s
is defined by $\omega\left(\left(i,j\right)\right):=\left(i,\left[j+3\right]_{6}\right)$
where $\left[t\right]_{6}$ denotes the remainder modulo $6$ of $t$.
\end{lem}

\begin{proof}
Since $e_{i,j}:=e_{i}f_{j}$ the equalities involving $a,b$ follows
analogously to \eqref{eq:TaftId}.

We compute
\[
f_{j}x^{m}=\frac{1}{6}\sum_{i=0}^{5}\zeta^{ji}b^{i}x^{m}=x^{m}\frac{1}{6}\sum_{i=0}^{5}\zeta^{ji}\left(-1\right)^{im}b^{i}=x^{m}\frac{1}{6}\sum_{i=0}^{5}\zeta^{\left(j+3m\right)i}b^{i}=x^{m}f_{j+3m}
\]
 and hence

\[
e_{i,j}x^{m}=e_{i}f_{j}x^{m}=e_{i}x^{m}f_{j+3m}=x^{m}e_{i}f_{j+3m}=x^{m}e_{i,j+3m}.
\]
Moreover
\begin{align*}
e_{i,j}g^{m} & =e_{i}f_{j}b^{3m}=e_{i}\frac{1}{6}\sum_{i=0}^{5}\zeta^{ji}b^{i+3m}\\
 & =\zeta^{-3jm}e_{i}\frac{1}{6}\sum_{i=0}^{5}\zeta^{j\left(i+3m\right)}b^{i+3m}=\zeta^{-3jm}e_{i,j}=\left(-1\right)^{jm}e_{i,j}.
\end{align*}
We have
\begin{align*}
S\left(f_{j}\right) & =\frac{1}{6}\sum_{i=0}^{5}\zeta^{ji}S\left(b\right)^{i}=\frac{1}{6}\sum_{i=0}^{5}\zeta^{ji}\left(e_{0}b^{-i}+e_{1}b^{i}\right)\\
 & =e_{0}\frac{1}{6}\sum_{i=0}^{5}\zeta^{ji}b^{-i}+e_{1}\frac{1}{6}\sum_{i=0}^{5}\zeta^{ji}b^{i}=e_{0}f_{-j}+e_{1}f_{j}.
\end{align*}
so that $S\left(e_{i,j}\right)=\left(e_{0}f_{-j}+e_{1}f_{j}\right)e_{i}=e_{i,\left(-1\right)^{i+1}j}.$
Note that $gx=b^{3}x=-xb^{3}=qxg$ so that the subalgebra of $A$
generated by $g$ and $x$ is a Taft algebra. As a consequence $S\left(x^{m}\right)$
is as in \eqref{eq:antipode} i.e.
\begin{align*}
S\left(x^{m}\right) & =\left(-1\right)^{m}q^{\frac{-m\left(m+1\right)}{2}}x^{m}g^{-m}=\left(-1\right)^{m}\left(-1\right)^{\frac{-m\left(m+1\right)}{2}}x^{m}g^{m}\\
 & =\left(-1\right)^{\frac{m\left(m-1\right)}{2}}x^{m}g^{m}=x^{m}g^{m}
\end{align*}
as $0\leq m\leq1$. We also have
\[
S^{2}\left(x^{m}\right)=S\left(x^{m}g^{m}\right)=S\left(g^{m}\right)S\left(x^{m}\right)=g^{-m}x^{m}g^{m}=\left(-1\right)^{m}x^{m}
\]
where the last equality holds since $0\leq m\leq1$ again.
\end{proof}
\begin{lem}
We have that $\int_{l}\left(A^{*}\right)=\Bbbk\lambda$ where
\begin{align*}
\lambda\left(x^{m}a^{i}b^{j}\right) & =\delta_{m,1}\delta_{i,0}\delta_{j,3}^{\equiv_{6}}\text{ for all }m\in\mathbb{N},i,j\in\mathbb{Z},\\
\lambda\left(x^{m}e_{i,j}\right) & =\frac{\left(-1\right)^{j}}{12}\delta_{m,1}\text{ for all }m\in\mathbb{N},i,j\in\mathbb{Z}.
\end{align*}
\end{lem}

\begin{proof}
We compute

\begin{align*}
 & \Delta\left(x^{m}a^{i}b^{j}\right)=\Delta\left(x\right)^{m}\Delta\left(a\right)^{i}\Delta\left(b\right)^{j}\\
 & =\left(\sum_{t=0}^{m}\binom{m}{t}_{-1}\left(x\otimes1\right)^{m-t}\left(b^{3}\otimes x\right)^{t}\right)\left(a^{i}\otimes a^{i}\right)\left(b^{j}\otimes e_{0}b^{j}+b^{-j}\otimes e_{1}b^{j}\right)\\
 & =\sum_{t=0}^{m}\binom{m}{t}_{-1}\left(x^{m-t}a^{i}b^{3t+j}\otimes x^{t}a^{i}e_{0}b^{j}\right)+\sum_{t=0}^{m}\binom{m}{t}_{-1}\left(x^{m-t}a^{i}b^{3t-j}\otimes x^{t}a^{i}e_{1}b^{j}\right)\\
 & =\sum_{t=0}^{m}\binom{m}{t}_{-1}\left(x^{m-t}a^{i}b^{3t+j}\otimes x^{t}e_{0}b^{j}\right)+\left(-1\right)^{i}\sum_{t=0}^{m}\binom{m}{t}_{-1}\left(x^{m-t}a^{i}b^{3t-j}\otimes x^{t}e_{1}b^{j}\right)
\end{align*}

Let us check that $\lambda$ as in the statement is a left integral
in $A^{*}$. For $0\leq m,i\leq1$ and $0\leq j\leq5$, we have
\begin{align*}
\sum\left(x^{m}a^{i}b^{j}\right)_{1}\lambda\left(\left(x^{m}a^{i}b^{j}\right)_{2}\right) & =\sum_{t=0}^{m}\binom{m}{t}_{-1}x^{m-t}a^{i}b^{3t+j}\lambda\left(x^{t}e_{0}b^{j}\right)\\
 & \quad+\left(-1\right)^{i}\sum_{t=0}^{m}\binom{m}{t}_{-1}x^{m-t}a^{i}b^{3t-j}\lambda\left(x^{t}e_{1}b^{j}\right)\\
 & =\sum_{t=0}^{m}\binom{m}{t}_{-1}x^{m-t}a^{i}b^{3t+j}\frac{1}{2}\delta_{t,1}\delta_{j,3}\\
 & \quad+\left(-1\right)^{i}\sum_{t=0}^{m}\binom{m}{t}_{-1}x^{m-t}a^{i}b^{3t-j}\frac{1}{2}\delta_{t,1}\delta_{j,3}\\
 & =\frac{1}{2}\binom{m}{1}_{-1}x^{m-1}a^{i}\delta_{j,3}+\frac{1}{2}\left(-1\right)^{i}\binom{m}{1}_{-1}x^{m-1}a^{i}\delta_{j,3}\\
 & =\frac{1+\left(-1\right)^{i}}{2}\left(m\right)_{-1}x^{m-1}a^{i}\delta_{j,3}\\
 & =\frac{1+\left(-1\right)^{i}}{2}\frac{1-\left(-1\right)^{m}}{2}x^{m-1}a^{i}\delta_{j,3}\\
 & \overset{i,m\leq1}{=}\delta_{i,0}\delta_{m,1}x^{m-1}a^{i}\delta_{j,3}=\delta_{i,0}\delta_{m,1}\delta_{j,3}=\lambda\left(x^{m}a^{i}b^{j}\right).
\end{align*}
Hence $\lambda$ is a left integral in $A^{*}$. We have observed
that $\int_{l}\left(A^{*}\right)$ is one-dimensional so that $\int_{l}\left(A^{*}\right)=\Bbbk\lambda$
. We compute
\begin{align*}
\lambda\left(x^{m}e_{i,j}\right) & =\lambda\left(x^{m}e_{i}f_{j}\right)=\sum_{s=0}^{1}\sum_{t=0}^{5}\frac{\left(-1\right)^{s}}{12}\zeta^{jt}\lambda\left(x^{m}a^{s}b^{t}\right)\\
 & =\sum_{s=0}^{1}\sum_{t=0}^{5}\frac{\left(-1\right)^{s}}{12}\zeta^{jt}\delta_{m,1}\delta_{s,0}\delta_{t,3}=\frac{1}{12}\zeta^{3j}\delta_{m,1}=\frac{\left(-1\right)^{j}}{12}\delta_{m,1}.
\end{align*}
\end{proof}
\begin{lem}
\label{lem:scalar-2}We have
\[
\left\langle x^{m}e_{i,j},x^{n}e_{u,v}\right\rangle =\frac{\left(-1\right)^{j\left(m+1\right)}}{12}\delta_{u,i}^{\equiv_{2}}\delta_{v,\left(-1\right)^{i+1}j}^{\equiv_{6}}\delta_{m+n,1},
\]

\[
\left\langle e_{i,j}x^{m},e_{u,v}x^{n}\right\rangle =\frac{\left(-1\right)^{j\left(m+1\right)}}{12}\delta_{u,i}^{\equiv_{2}}\delta_{v,\left(-1\right)^{i+1}j-3}^{\equiv_{6}}\delta_{m+n,1}.
\]

In particular the form $\left\langle -,-\right\rangle $ is as in
\ref{eq:formcan}, where
\[
\mu\left(\left(i,j\right)\right):=\left(i,\left[\left(-1\right)^{i+1}j-3\right]_{6}\right),\quad\nu\left(m\right):=1-m\quad\text{and}\quad d_{\left(i,j,m\right)}:=\frac{\left(-1\right)^{j\left(m+1\right)}}{12}.
\]
 As a consequence the form $\left\langle -,-\right\rangle $ is monomial
with respect to the basis $\mathcal{B}=\left\{ e_{i,j}x^{m}\mid0\leq s,m\leq N-1\right\} $.
Moreover the Nakayama isomorphism $\gamma$ is given by $\gamma\left(h\right)=hg$,
for every $h\in A$.
\end{lem}

\begin{proof}
We compute
\begin{align*}
x^{n}e_{u,v}S\left(x^{m}e_{i,j}\right) & =x^{n}e_{u,v}S\left(e_{i,j}\right)S\left(x^{m}\right)\\
 & =x^{n}e_{u,v}e_{i,\left(-1\right)^{i+1}j}S\left(x^{m}\right)\\
 & =x^{n}e_{u,v}e_{i,\left(-1\right)^{i+1}j}x^{m}g^{m}\\
 & =\delta_{u,i}^{\equiv_{2}}\delta_{v,\left(-1\right)^{i+1}j}^{\equiv_{6}}x^{n}e_{u,v}x^{m}g^{-m}\\
 & =\delta_{u,i}^{\equiv_{2}}\delta_{v,\left(-1\right)^{i+1}j}^{\equiv_{6}}x^{n}x^{m}e_{u,v+3m}g^{-m}\\
 & =\left(-1\right)^{\left(v+3m\right)m}\delta_{u,i}^{\equiv_{2}}\delta_{v,\left(-1\right)^{i+1}j}^{\equiv_{6}}x^{m+n}e_{u,v+3m}
\end{align*}

Thus we have
\begin{align*}
\left\langle x^{m}e_{i,j},x^{n}e_{u,v}\right\rangle  & =\lambda\left(x^{n}e_{u,v}S\left(x^{m}e_{i,j}\right)\right)\\
 & =\left(-1\right)^{\left(v+3m\right)m}\delta_{u,i}^{\equiv_{2}}\delta_{v,\left(-1\right)^{i+1}j}^{\equiv_{6}}\lambda\left(x^{m+n}e_{u,v+3m}\right)\\
 & =\frac{\left(-1\right)^{v+3m}}{12}\left(-1\right)^{\left(v+3m\right)m}\delta_{u,i}^{\equiv_{2}}\delta_{v,\left(-1\right)^{i+1}j}^{\equiv_{6}}\delta_{m+n,1}\\
 & =\frac{\left(-1\right)^{v\left(m+1\right)}}{12}\delta_{u,i}^{\equiv_{2}}\delta_{v,\left(-1\right)^{i+1}j}^{\equiv_{6}}\delta_{m+n,1}\\
 & =\frac{\left(-1\right)^{j\left(m+1\right)}}{12}\delta_{u,i}^{\equiv_{2}}\delta_{v,\left(-1\right)^{i+1}j}^{\equiv_{6}}\delta_{m+n,1}.
\end{align*}

We also have
\begin{align*}
\left\langle e_{i,j}x^{m},e_{u,v}x^{n}\right\rangle  & =\left\langle x^{m}e_{i,j+3m},x^{n}e_{u,v+3n}\right\rangle \\
 & =\frac{\left(-1\right)^{\left(j+3m\right)\left(m+1\right)}}{12}\delta_{u,i}^{\equiv_{2}}\delta_{v+3n,\left(-1\right)^{i+1}\left(j+3m\right)}^{\equiv_{6}}\delta_{m+n,1}\\
 & =\frac{\left(-1\right)^{j\left(m+1\right)}}{12}\delta_{u,i}^{\equiv_{2}}\delta_{v+3n,\left(-1\right)^{i+1}j+\left(-1\right)^{i+1}3m}^{\equiv_{6}}\delta_{m+n,1}\\
 & =\frac{\left(-1\right)^{j\left(m+1\right)}}{12}\delta_{u,i}^{\equiv_{2}}\delta_{v+3n,\left(-1\right)^{i+1}j-3m}^{\equiv_{6}}\delta_{m+n,1}\\
 & =\frac{\left(-1\right)^{j\left(m+1\right)}}{12}\delta_{u,i}^{\equiv_{2}}\delta_{v,\left(-1\right)^{i+1}j-3}^{\equiv_{6}}\delta_{m+n,1}.
\end{align*}

By the foregoing it is clear that as in \ref{eq:formcan}, where $\mu\left(\left(i,j\right)\right):=\left(i,\left[\left(-1\right)^{i+1}j-3\right]_{6}\right)$,
$\nu\left(m\right):=1-m$ and $d_{\left(i,j,m\right)}:=\frac{\left(-1\right)^{j\left(m+1\right)}}{12}$.

Thus, by Lemma \ref{lem:scalarcan}, the form $\left\langle -,-\right\rangle $
is monomial with respect to the basis $\mathcal{B}=\left\{ e_{i,j}x^{m}\mid0\leq s,m\leq N-1\right\} $.
Moreover the Nakayama isomorphism $\gamma$ is given by $\gamma\left(e_{i,j}x^{m}\right)=\frac{d_{\left(i,j,m\right)}}{d_{\left(\mu\left(\left(i,j\right)\right),\nu\left(m\right)\right)}}e_{\mu^{2}\left(i,j\right)}x^{\nu^{2}\left(m\right)}$
for all $i,j,m$. Since $\mu^{2}=\mathrm{Id}=\nu^{2}$, we obtain
$\gamma\left(e_{i,j}x^{m}\right)=\frac{d_{\left(i,j,m\right)}}{d_{\left(\mu\left(\left(i,j\right)\right),\nu\left(m\right)\right)}}e_{i,j}x^{m}$.
We compute
\[
\frac{d_{\left(i,j,m\right)}}{d_{\left(\mu\left(\left(i,j\right)\right),\nu\left(m\right)\right)}}=\frac{\frac{\left(-1\right)^{j\left(m+1\right)}}{12}}{\frac{\left(-1\right)^{\left(\left(-1\right)^{i+1}j-3\right)\left(1-m+1\right)}}{12}}=\frac{\left(-1\right)^{j\left(m+1\right)+\left(-1\right)^{i+1}jm-3m}}{12}=\left(-1\right)^{j+m}
\]
Thus
\[
\gamma\left(e_{i,j}x^{m}\right)=\frac{d_{\left(i,j,m\right)}}{d_{\left(\mu\left(\left(i,j\right)\right),\nu\left(m\right)\right)}}e_{i,j}x^{m}=\left(-1\right)^{j+m}e_{i,j}x^{m}=\left(-1\right)^{m}e_{i,j}gx^{m}=e_{i,j}x^{m}g
\]
 and hence $\gamma\left(h\right)=hg$ for every $h\in H$.
\end{proof}
\begin{lem}
\label{lem:lambdaS-other}In the setting of Proposition \ref{prop:action},
take $V$ the algebra $A$ as above with left regular action. Then
$\vartriangleleft=\blacktriangleleft$ and, for every $x\in H,r\in R$,
we have $x\vartriangleleft r=S\left(r\right)x.$
\end{lem}

\begin{proof}
By Lemma \ref{lem:scalar-2}, we have that the map $\gamma$ is given
by $\gamma\left(h\right)=hg$ for every $h\in H$. Thus $\gamma$
is left $A$-linear. By Proposition \ref{prop:action}, we have $\vartriangleleft=\blacktriangleleft$
.
\end{proof}
By \eqref{eq:lambdagamma} we have
\begin{align*}
\lambda S\left(x^{m}e_{s,t}\right) & =\lambda\gamma\left(x^{m}e_{s,t}\right)=\lambda\left(x^{m}e_{s,t}g\right)\\
 & =\left(-1\right)^{t}\lambda\left(x^{m}e_{s,t}\right)=\left(-1\right)^{t}\frac{\left(-1\right)^{t}}{12}\delta_{m,1}=\frac{1}{12}\delta_{m,1}.
\end{align*}

We compute
\begin{align*}
\lambda\left(x^{u}x^{m}e_{s,t}\right) & =\lambda\left(x^{u+m}e_{s,t}\right)=\frac{\left(-1\right)^{t}}{12}\delta_{u+m,1}\\
= & \left(-1\right)^{u}\frac{\left(-1\right)^{t+3u}}{12}\delta_{u+m,1}=\left(-1\right)^{u}\lambda\left(x^{u+m}e_{s,t+3u}\right)\\
= & \left(-1\right)^{u}\lambda\left(x^{m}x^{u}e_{s,t+3u}\right)=\lambda\left(x^{m}e_{s,t}\left(-1\right)^{u}x^{u}\right)=\lambda\left(x^{m}e_{s,t}S^{2}\left(x^{u}\right)\right)
\end{align*}
which implies $\lambda\left(rh\right)=\lambda\left(hS^{2}\left(r\right)\right)$
for every $r\in R,h\in A$. As a consequence, we get $x\vartriangleleft r=S\left(r\right)x$
as in the proof of Lemma \ref{lem:lambdaS}.

We are now able to compute the indecomposable ideals and their orthogonals.
\begin{thm}
\label{thm:indecomposableCDMM}Consider the Hopf algebra
\[
A=\Bbbk\left\langle a,b,x\mid a^{2}=1=b^{6},x^{2}=0,ab=ba,ax=xa,bx=-xb\right\rangle
\]
and let $R$ be the subalgebra of $A$ generated by $x$. For $0\leq m,s\leq1,0\leq t\leq5$
set $N_{s,t,m}:=e_{s,t}x^{m}R$. Then the $N_{s,t,m}$'s form an irredundant
set of representatives of $\mathcal{L}\left(A_{A}\right)$ and
\[
\mathcal{L}_{\mathrm{in}}\left(A_{A}\right)=\left\{ \left(1+rx\right)N_{s,t,m}\mid r\in R,0\leq m,s\leq1,0\leq t\leq5\right\} .
\]
\end{thm}

\begin{proof}
By Lemma \ref{lem:formuTaft} $A$ is of the form $\Bbbk\left(\omega,N\right)$
as in Section \ref{sec:Indec-ideals}. Thus the statement follows
by Theorem \ref{thm:indecomposableGold}.
\end{proof}
\begin{thm}
\label{thm:orthogonal-CDMM}Let $a=a\left(x\right)\in R$ be an invertible
element (we can assume $a\left(0\right)=1)$, $s,t,m$ integers such
that $0\leq m,s\leq1,0\leq t\leq5$; then we have
\begin{align*}
N_{s,t,0}^{\perp} & =\left(\bigoplus_{j}N_{1-s,j,0}\right)\oplus\left(\bigoplus_{\substack{j\not\equiv_{6}\left(-1\right)^{s+1}t+3}
}N_{s,j,0}\right),\\
N_{s,t,m}^{\perp} & =N_{s,t,0}^{\perp}\oplus N_{s,\left(-1\right)^{s+1}t+3,2-m},\\
\left(aN_{s,t,m}\right)^{\perp} & =S\left(a^{-1}\right)N_{s,t,m}^{\perp}.
\end{align*}
\end{thm}

\begin{proof}
By Lemma \ref{lem:lambdaS-other} and \eqref{eq:orthNleft},\eqref{eq:orthNright},
we have that $\left(aN_{s,t,m}\right)^{\perp_{L}}=S\left(a^{-1}\right)N_{s,t,m}^{\perp_{L}}$
and $\left(aN_{s,t,m}\right)^{\perp_{R}}=S\left(a^{-1}\right)N_{s,t,m}^{\perp_{R}}$
for every $s,t,m$. We know that $\mu\left(\left(i,j\right)\right):=\left(i,\left[\left(-1\right)^{i+1}j-3\right]_{6}\right)=\left(i,\left[\left(-1\right)^{i+1}j+3\right]_{6}\right)$
(so that $\mu^{2}=\mathrm{Id}$) and $\nu\left(m\right):=1-m$ in
our case. Thus by Theorem \ref{thm:orthog-indec}, we get $N_{s,t,m}^{\perp_{R}}=N_{s,t,m}^{\perp_{L}}$
for every $s,t,m$ (hence we can use the notation $\perp$) and the
equalities in the present statement holds as $\left(i,j\right)\neq\mu\left(\left(s,t\right)\right)=\left(s,\left[\left(-1\right)^{s+1}t+3\right]_{6}\right)$
means either $i=1-s$ or $j\not\equiv_{6}\left(-1\right)^{s+1}t+3$.
\end{proof}

\end{document}